\documentclass[11pt]{article}
\usepackage{amsmath}
\usepackage{amsthm}
\usepackage{amsfonts}
\usepackage[usenames,dvipsnames]{xcolor}
\usepackage{enumitem}
\usepackage[T1]{fontenc}
\usepackage[utf8]{inputenc}
\newtheorem{theo}{Theorem}[section]
\newtheorem{lem}[theo]{Lemma}
\newtheorem{prop}[theo]{Proposition}

\newtheorem{cor}[theo]{Corollary}

\newtheorem{defi}[theo]{Definition}

\newcommand{\mysection}[1]{\section{#1} \setcounter{equation}{0}}
\newcommand{\proofc}{{\sc Proof} \ }
\newcommand{\be}{\begin{equation} \label}
\newcommand{\ee}{\end{equation}}
\newcommand{\bea}{\begin{eqnarray}\label}
\newcommand{\eea}{\end{eqnarray}}
\newcommand{\bas}{\begin{eqnarray*}}
\newcommand{\eas}{\end{eqnarray*}}
\newcommand{\bit}{\begin{itemize}}
\newcommand{\eit}{\end{itemize}}
\renewcommand{\qed}{\hfill$\Box$ \vskip.2cm}
\newcommand{\nn}{\nonumber}
\newcommand{\R}{\mathbb{R}}
\newcommand{\N}{\mathbb{N}}

\newcommand{\eps}{\varepsilon}

\newcommand{\supp}{{\rm supp} \, }

\newcommand{\io}{\int_\Omega}

\newcommand{\abs}{\\[5pt]}

\newcommand{\tv}{\widetilde{v}}
\newcommand{\ov}{\overline{v}}
\newcommand{\uv}{\underline{v}}
\newcommand{\uz}{\underline{z}}

\newcommand{\ir}{\int_{\R^n}}
\newcommand{\irn}{\int_{\R^n}}
\newcommand{\sst}{s_\star}
\newcommand{\ssst}{s_{\star\star}}
\newcommand{\E}{{\mathcal{E}}}
\newcommand{\uu}{\underline{u}}
\newcommand{\uo}{\overline{u}}
\newcommand{\phii}{\varphi}

\newcommand{\na}{\nabla}

\newcommand{\norm}[2][]{\left\|#2\right\|_{#1}}
\newcommand{\calL}{\mathcal{L}}
\newcommand{\set}[1]{\left\{#1\right\}}
\newcommand{\Bbar}{\overline{B}}

\newcommand{\Rst}{R_{\star}}
\newcommand{\Ctilde}{\widetilde{C}}
\newcommand{\betatilde}{B} %beta im Nenner sah mit \tilde unleserlich aus
\newcommand{\alphatilde}{A}

\begin{document}
\enlargethispage{10mm}
\title{Counterintuitive dependence of temporal asymptotics on initial decay 
in a nonlocal degenerate parabolic equation arising in game theory}
% Oder: \title{Counterintuitive implications of initial decay on temporal asymptotics 
% in a nonlocal degenerate parabolic equation arising in game theory} ???
%
%
%
\author{
Johannes Lankeit\footnote{jlankeit@math.uni-paderborn.de}\\
{\small Institut f\"ur Mathematik, Universit\"at Paderborn,}\\
{\small 33098 Paderborn, Germany} 
\and
Michael Winkler\footnote{michael.winkler@math.uni-paderborn.de}\\
{\small Institut f\"ur Mathematik, Universit\"at Paderborn,}\\
{\small 33098 Paderborn, Germany} }
%
%
%\date{} %Der Befehl kann am Ende gerne wieder rein, vorerst ist es f\"ur mich sehr hilfreich zu sehen, wie aktuell die Version jeweils ist, die ich gerade vor mir habe.
\maketitle
\begin{abstract}
\noindent 
  We consider the degenerate parabolic equation with nonlocal source given by
  \bas
	u_t=u\Delta u +  u \io |\nabla u|^2,
  \eas
which has been proposed as model for the evolution of the density distribution of frequencies with 
which different strategies are pursued in a population obeying the rules of replicator dynamics in a continuous infinite-dimensional setting.\abs
  Firstly, for all positive initial data 
  $u_0\in C^0(\R^n)$ satisfying $u_0\in L^p(\R^n)$ for some $p\in (0,1)$
  as well as $\ir u_0=1$, the corresponding Cauchy problem in $\R^n$ is seen to possess a global positive classical solution 
  with the property that $\ir u(\cdot,t)=1$ for all $t>0$. \abs
  Thereafter, the main purpose of this work consists in reavealing 
  a dependence of the large time behavior of these solutions
  on the spatial decay of the initial data in a direction that seems unexpected
  when viewed against the background of known behavior in large classes of scalar parabolic problems.
  In fact, it is shown that all considered solutions asymptotically decay with respect to their spatial $H^1$ norm, so that
  \bas
	\E(t):=\int_0^t \ir |\nabla u(\cdot,t)|^2, \qquad t>0,
  \eas
  always grows in a significantly sublinear manner in that
  \be{a1}
	\frac{\E(t)}{t} \to 0
	\qquad \mbox{as } t\to\infty;
  \ee
  the precise growth rate of $\E$, however, depends on the initial data in such a way that fast decay rates of $u_0$
  enforce rapid growth of $\E$. 
  To this end, examples of algebraical and certain exponential types of initial decay are detailed, inter alia
  generating logarithmic and arbitrary sublinear algebraic growth rates of $\E$, and moreover indicating
  that (\ref{a1}) is essentially optimal.\abs
\noindent {\bf Key words:} degenerate diffusion; nonlocal source; decay rate\\
  {\bf Math Subject Classification (2010):} 35B40 (primary), 35K55, 35K65, 91A22 (secondary)
\end{abstract}
%  35B40 View Publications (1973-now) Asymptotic behavior of solutions 
%  35K65 View Publications (1980-now) Degenerate parabolic equations 
%  35K55 View Publications (1973-now) Nonlinear parabolic equations 
%  35A01 View Publications (2010-now) Existence problems: global existence, local existence, non-existence 
%  91A22 View Publications (2000-now) Evolutionary games 
%
%
%
%
\newpage
\section{Introduction}\label{intro}
For $n\ge 1$ and given positive initial data $u_0\in C^0(\R^n)$, we consider positive solutions of the Cauchy problem 
\be{0}
	\left\{ \begin{array}{ll}
	u_t=u\Delta u + u\ir |\nabla u|^2, \qquad & x\in\R^n, \ t>0, \\[1mm]
	u(x,0)=u_0(x), & x\in\R^n,
	\end{array} \right.
\ee
which has been proposed in the context of evolutionary game theory as a model for strategies which are pursued within a population and the distribution of which evolves in accordance with the rules of replicator dynamics. In this model, the continuum of possible strategies is given by the domain $\R^n$ and $u$ plays the role of the density of the relative frequency with which the strategies are followed. 
The analogue for a setting involving finitely many strategies goes back to the works of Taylor and Jonker \cite{taylor_jonker} and Maynard Smith \cite{maynard_smith}, and infinite-dimensional variants have been treated in \cite{bomze} and \cite{oechssler_riedel}. Building on these, for steep payoff kernels of Gaussian type, in \cite{KPY08,KPXY10} %genauer: KPY08 hat die Gleichung zuerst eingeführt (PS09 betrachtet sie übrigens auch); eine über ``da sieht man leichter was als in der ganz allgemeinen Form'' hinausgehende Begründung, wo diese Form interessant sein könnte, gibt's erst in KPXY10
the PDE \eqref{0} was introduced. For additional details concerning the modeling background, we refer to \cite[Appendix A]{klw}, \cite{KPXY10}, and the references therein.\abs
Mathematically, (\ref{0}) can be viewed as joining two mechanisms which are quite delicate even when regarded separately:
Firstly, the PDE therein contains a reaction term which is superlinear, and hence potentially destabilizing
in the extreme sense of possibly enforcing blow-up, and which is moreover nonlocal and thereby may destroy 
any ordering property, as constituting a well-appreciated feature of parabolic equations with exclusively local terms
(cf.~the book \cite{quittner_souplet} for a large variety of aspects related to this). 
Secondly, the diffusion process in (\ref{0}) degenerates near points where $u$ is small, and it is indicated by the
analysis of the simple diffusion equation 
\be{up}
	u_t=u^p\Delta u
\ee 
that with regard to this degeneracy, (\ref{0}) is
precisely critical: Namely, whereas in the case $p\in (0,1)$ the equation
(\ref{up}) actually reduces to a porous medium equation with its well-developed 
theory on existence, uniqueness and (H\"older) regularity of weak solutions (\cite{aronson}), 
it is known that stronger degeneracies in (\ref{up}) may bring about much more irregular and unexpected solution
behavior (see e.g.~\cite{win_osc} or \cite{win_homoclinic} for two recent examples); 
the particular borderline situation of the exponent $p=1$ is indicated e.g.~by the classical
results from \cite{bertsch} and \cite{luckhaus_dalpasso}) which assert that precisely 
for $p\ge 1$, weak solutions to (\ref{up})
can spontaneously develop disconinuities and in general are not unique.\abs
Accordingly, previous studies on the PDE in (\ref{0}) 
either concentrate on the construction of particular
radially symmetric self-similar solutions of self-similar structure, thus actually concerned with a corresponding
ODE analysis (\cite{KPY08}), or resort to appropriate generalized solution frameworks, as recently done
in \cite{klw} for the homogeneous Dirichlet problem associated with (\ref{0}) in bounded domains $\Omega\subset\R^n$.
Beyond a corresponding local existence statement, the latter work moreover provides some rigorous evidence for the 
mass evolution property
\be{evol_mass}
	\frac{d}{dt} \io u = \bigg\{ \io |\nabla u|^2 \bigg\} \cdot \bigg\{ \io u - 1 \bigg\},
\ee
formally satisfied by solutions to the considered problem, in the sense of the implication that if
the respective initial data satisfy $\io u_0 < 1$, a global weak solution $u$ can be found which is such that
$\io u(\cdot,t)< 1$ for all $t>0$ and such that $u(\cdot,t)\to 0$ as $t\to\infty$ in an appropriate sense, 
while whenever $\io u_0>1$, a solution can be constructed which blows up in finite time $T$ with respect to its spatial
$L^\infty$ norm, and for which we have $\io u(\cdot,t)>1$ for all $t\in (0,T)$.
For this Dirichlet problem, an essentially complete understanding has also been achieved in 
the critical case when $\io u_0=1$:	%, evidently being of particular interest since 
%in the present modeling context the unknown $u$ represents a probability density: 
For such initial data, namely,
the initial-boundary value problem in question possesses a global generalized solution which satisfies
$\io u(\cdot,t)=1$ for all $t>0$ and furthermore stabilizes toward the solution $\varphi$ of the Dirichlet problem
for $-\Delta \varphi=1$ in $\Omega$ in the large time limit (\cite{lankeit}).\abs
{\bf Global existence of classical solutions.}\quad
Bearing in mind that in the present modeling context the unknown $u$ represents a probability density, 
in this work we shall focus on the analogue of the latter unit-mass situation in the Cauchy problem (\ref{0}),
thus concentrating on solutions satisfying the additional condition
\be{unit}
	\ir u(\cdot,t)=1
\ee
throughout evolution. 
Since this in particular requires solutions to attain small values in large spatial regions, even the mere construction
of solutions seems to be nontrivial and not to be achievable through a straightforward approximation
by solutions to homogeneous Dirichlet problems in balls with increasing size, for instance; 
in fact, it seems unclear whether such approximate solutions enjoy suitable compactness properties which are sufficient
to ensure convergence to a solution of (\ref{0}) additionally satisfying (\ref{unit}), rather than e.g.~a 
corresponding inequality only.\abs
We first address this problem of solvability, and it will turn out that despite these obstacles, for a large class of
initial data a global solution can actually be found in the classical framework specified as follows.
\begin{defi}\label{defi11}
  Let $u_0\in C^0(\R^n)$ be positive. Then by a {\em positive classical solution} of (\ref{0}) in $\R^n\times (0,\infty)$
  we mean a positive function
  \bas
	u\in C^0(\R^n\times [0,\infty)) \cap C^{2,1}(\R^n\times (0,\infty))
  \eas
  which is such that $\nabla u(\cdot,t)$ belongs to $L^2(\R^n)$ for all $t>0$, that
  \bas
	(0,\infty) \ni t \mapsto \ir |\nabla u(\cdot,t)|^2
	\quad \mbox{is continuous,}
  \eas
  and which satisfies both identities in (\ref{0}) in the pointwise sense.
\end{defi}
Within this setting, the first of our main results asserts global solvability under mild assumptions
on the spatial decay of $u_0$.
\begin{theo}\label{theo10}
  Assume that $u_0 \in C^0(\R^n)$ is positive and satisfies
  \be{10.1}
	u_0\in L^p(\R^n) 
	\qquad \mbox{for some } p\in (0,1)
  \ee
  as well as 
  \be{10.2}
	u_0(x)\to 0
	\qquad \mbox{as } |x|\to\infty
  \ee
  and
  \be{10.3}
	\ir u_0	=1.
  \ee
  Then the problem (\ref{0}) possesses at least one positive classical solution $u$ 
  in $\R^n\times (0,\infty)$ fulfilling
  \be{mass_u}
	\ir u(\cdot,t)=1
	\qquad \mbox{for all } t>0.
  \ee
\end{theo}
{\bf Large time behavior.} \quad
Our next objective consists in deriving some information on the large time behavior of solutions, where it will turn
out that a convenient description thereof can be given in terms of the cumulated free energy functional $\E$ given by
\be{def_E}
	\E(t) := \int_0^t \ir |\nabla u|^2, \qquad t>0;
\ee
indeed, from (\ref{0}) it may be expected that suitably smooth positive solutions with appropriate spatial decay satisfy
the identity
\be{energy}
	\ir \frac{u_t^2}{u} 
	+ \frac{1}{2} \frac{d}{dt} \ir |\nabla u|^2
	= \bigg\{ \ir |\nabla u|^2 \bigg\} \cdot \bigg\{ \frac{d}{dt} \ir u \bigg\},
\ee
whence the Dirichlet integral $\ir |\nabla u|^2$ should in fact play the role of a Lyapunov functional on the set
of solutions to (\ref{0}) with conserved unit mass.
A rigorous partial verification of this is provided by the following proposition which, beyond a correspondingly
expected monotonicty property, asserts decay of this free energy and thereby states a first
asymptotic feature of $\E$ which holds without imposing further restrictions on the initial data
apart from those in Theorem \ref{theo10}.
\begin{prop}\label{prop103}
  Let $u_0\in C^0(\R^n)$ be positive and such that (\ref{10.1})-(\ref{10.3}) hold,
  and let $u$ denote the corresponding positive classical solution of (\ref{0}) obtained in Theorem \ref{theo10}.
  Then $(0,\infty)\ni t\mapsto \io |\nabla u(\cdot,t)|^2$ is nonincreasing with
  \be{103.2}
	\ir |\nabla u(\cdot,t)|^2 \to 0
	\qquad \mbox{as } t\to\infty.
  \ee
  In particular, with $\E$ as in (\ref{def_E}) we have
  \be{103.3}
	\frac{\E(t)}{t} \to 0
	\qquad \mbox{as } t\to\infty.
  \ee
\end{prop}
As already suggested by well-known properties of the mere heat equation,
but also confirmed in numerous cases of Cauchy problems for nonlinear diffusion equations 
(see \cite{vazquez_book} and also \cite{fvwy} and the references therein for some recent developments) and 
nonlinearly forced semilinear parabolic equations (\cite{lee_ni}, \cite{fkwy_linear}, \cite{quittner_souplet}),
more detailed asymptotic properties of solutions may depend on the spatial decay of the initial data in quite a colorful
manner.
The most commonly found flavor of	%types of 		%semi-quantitative 
results in this direction is that in the respective problem	%detect	%, in a more or less detailed manner, 
a certain monotone dependence, either continuous or dicontinuous, of temporal
convergence rates on rates of initially present spatial decay is detected, where 
in apparently perfect accordance with intuitive ideas, monotonicity is usually directed in such a way that
fast spatial decay implies fast temporal decay and vice versa.\abs
In the present context, however, our results indicate that 
the unit-mass constraint considered here
enforces a very subtle balance between nonlocal reaction and nonlinear diffusion in (\ref{0}) which is such that
this direction of monotone dependence is actually reversed.\abs
For initial data satisfying algebraic decay conditions consistent with our overall integrability requirements,
this becomes manifest in a corresponding counterintuitive ordering of the respective coefficients in logarithmic
growth estimates for $\E$ which are the objective of the following two statements.
\begin{theo}\label{theo200}
  Suppose that $u_0\in C^0(\R^n)$ is positive such that (\ref{10.1})-(\ref{10.3}) are fulfilled,
  and let $u$ and $\E$ be as in Theorem \ref{theo10} and (\ref{def_E}).\abs
  i) \ If 
  \bas
	u_0(x)\geq c_0(1+|x|)^{-\gamma} \qquad \mbox{for all } x\in \R^n
  \eas
  for some $c_0>0$ and $\gamma>n$, then there exist $T>0$ and $C>0$ such that 
  \begin{equation}\label{eq:E_upperbd}
 	\E(t) \leq \frac{\gamma-n}{n+2} \ln t + C \qquad \mbox{for all } t>T.
  \end{equation}
  ii) \ If there exist $C_0>0$ and $\gamma>n$ fulfilling
  \begin{equation}\label{eq:u0leq}
	u_0(x)\leq C_0 (1+|x|)^{-\gamma} \qquad \mbox{ for all } x\in\R^n,
  \end{equation}
  then for any $\eps\in (0,\gamma)$ one can find 
  $C(\eps)>0$ with the property that
  \begin{equation}\label{eq:Egeq_cor}
	\E(t) \geq \frac{\gamma-n - \eps}{n+2} \ln t - C(\eps) \qquad \mbox{for all } t>0.
  \end{equation}
\end{theo}
A natural next problem appears to consist in determining how far it is possible to close the gap between the logarithmic
estimates from Theorem \ref{theo200}, thus essentially sharp e.g.~for initial data with precise
algebraic decay, and the general property (\ref{103.3}) which excludes linear and faster growth of $\E$ but
would be well consistent with any sublinearly increasing $\E$.
In order to obtain solutions for which $\E$ becomes substantially larger than expressed in Theorem \ref{theo200}, 
we shall consider more rapidly decaying intial data, some prototypical choices of 
which will be specified below in Corollary \ref{cor:firstexample} and Corollary \ref{cor:secondexample}.
For technical reasons, the formulation of our general statement in this direction will require the introduction of a 
function $\ell\in C^0([0,\infty))$ which is such that
\be{ell:zero_pos_nondec}
	\ell(0)=0, \quad
	\ell>0 \mbox{ in } (0,\infty)
	\quad \mbox{and} \quad
	\ell \mbox{ is nondecreasing on } (0,\infty),
\ee
but that $\ell$ grows in a significantly sublinear manner in that
\be{ell:monLstrich}
	\mbox{there exists $\xi_0>0$ such that } 
	(\xi_0,\infty) \ni \xi\mapsto \xi \cdot \ell^{\frac{n+2}n}\Big(\frac{1}{\xi}\Big) 
	\mbox{ is nondecreasing,}
\ee
and that
\be{ell:intcond}
  	\int_1^\infty \frac{d\xi}{\xi \cdot \ell^{\frac{n+2}{n}}\big(\frac{1}{\xi}\big)} =\infty
\ee
which, for instance, is satisfied for $\ell(\xi):=\xi^\alpha$, $\xi\ge 0$, whenever $\alpha\in (0,\frac{n}{n+2})$,
but also for $\ell(\xi):=\ln^\alpha (\xi+1)$, $\xi\ge 0$, for arbitrary $\alpha>0$.
Given any such $\ell$, we furthermore let
\be{eq:definecalL}
	 \calL(t):=\int_1^t \frac{d\xi}{\xi \cdot \ell^\frac{n+2}n\big(\frac{1}{\xi}\big)},
	\qquad  t>1,
\ee
noting that by \eqref{ell:zero_pos_nondec}, $\calL$ is strictly monotone and so is $\calL^{-1}\colon (0,\infty)\to(1,\infty)$, 
which is well-defined due to \eqref{ell:intcond}.  
Then under appropriate conditions on the spatial decay of the initial data, the quantity $\E(t)$ can be estimated
in terms of $\calL$.
\begin{theo}\label{theo300}\label{rapid}
  Let $u_0\in C^0(\R^n)$ be a positive function satisfing (\ref{10.1})-(\ref{10.3}),
  and let $u$ and $\E$ be as given by Theorem \ref{theo10} and (\ref{def_E}).
  Moreover, let $\ell \in C^0([0,\infty)$ satisfy (\ref{ell:zero_pos_nondec}), (\ref{ell:monLstrich}) and (\ref{ell:intcond}).%\abs
\begin{enumerate}[label={\roman*)},leftmargin=0em, itemindent=0.5cm]
\item \label{r_Eleq}
  Assume that there exists $\xi_1\in (0,\infty]$ such that $\ell$ is strictly increasing on $(0,\xi_1)$,
  and that
  \be{ell:0limit}
 	\frac{\xi \ell'(\xi)}{\ell(\xi)} \to 0
	\qquad \mbox{as } \xi\searrow 0.
  \ee
  Moreover, suppose that there exist $q\in (0,1)$ and 
  $\Rst>\max\set{\frac1{\sqrt[n]{\xi_1}},\frac1{\sqrt[n]{\lim_{\xi\nearrow \xi_1} \ell(\xi)}}}$ %schreibweise mit lim für den Fall \xi _1=\infty
  such that
  \bas
	u_0(x)\geq \bigg\{ \ell^{-1}\left(\frac{1}{|x|^n}\right)\bigg\}^q
	\qquad \mbox{ for all }x\in\R^n\setminus B_{\Rst}.
  \eas
  Then for some $t_0>0$ and $C_1 >0$, $C_2\geq 0$
   \bas
	 \E(t)\leq\ln\Big( (\calL^{-1})'(C_1  t)\Big) +C_2
	\qquad \mbox{for all }t>t_0.
  \eas
\item \label{r_Egeq}
  Suppose that there exist $\xi_2>0$, $a>0$ and $\lambda_0>0$ such that $\ell\in C^2((0,\xi_2))$ and 
  \be{ell3}
	\xi\ell''(\xi) \ge - \ell'(\xi)
	\qquad \mbox{for all } \xi\in (0,\xi_2)
  \ee
  as well as
  \be{ell33}
	\ell(\xi)\le (1+a\lambda) \ell(\xi^{1+\lambda}) 
	\qquad \mbox{ for all } \xi\in(0,\xi_2) \mbox{ and each } \lambda\in(0,\lambda_0).
  \ee
  If 
  there exists a radially symmetric $\uo_0\in C^0(\R^n)$, nonincreasing with respect to $|x|$ and satisfying $u_0\le \uo_0$ 
  as well as
  \bas
	\uo_0<\min\Big\{ \xi_2^2 \, , \, \xi_2^{\frac{2}{1+q_0}}\Big\} \qquad \mbox{in }\R^n
  \eas
  for some $q_0>0$ and 
  \bas
 	\ir \ell(\uo_0) <\infty,
  \eas
  then we can find $t_0>0$ and $C_1 >0$, $C_2\ge 0$ such that with $\E$ and $\calL$ taken from (\ref{def_E}) and (\ref{eq:definecalL})
  we have 
  \bas
  	\E(t)  \geq \ln \Big( (\calL^{-1})'(C_1  t)\Big) -C_2
	\qquad \mbox{for any }t>t_0.
  \eas
\end{enumerate}
\end{theo}
In particular, arbitrary sublinear algebraic growth of $\E$ occurs for initial data exhibiting certain exponential types
of decay.
Let us underline that due to the upward monotonicity of the mapping $0<\beta\mapsto \frac{1}{1+\frac{n+2}{n}\beta}$,
the following result is again consistent with the above observation that fast spatial decay of $u_0$
generates slow decay of the solution.
\begin{cor}\label{cor:firstexample}
  Let $u_0\in C^0(\R^n)$ be positive and such that (\ref{10.1})-(\ref{10.3}) hold, and let $u$ and $\E$
  be as in Theorem \ref{theo10} and (\ref{def_E}).\\
\begin{enumerate}[label={\roman*)},leftmargin=0em, itemindent=0.5cm]
\item\label{cor:firstexample:leq}
  If
  \bas
	u_0(x)\geq c_0 e^{-\alpha |x|^\beta}\qquad \text{for all } x\in\R^n\setminus B_{\Rst}
  \eas
  for some $\Rst>0$, $\beta>0$, $\alpha\in(0,1)$ and $c_0>0$, 
  then there exist $t_0>0$ and $C>0$ such that
  \bas
	\E(t)\leq Ct^{1/(1+\frac{n+2}\beta)} 	
	\qquad \mbox{for all } t>t_0.
  \eas
\item\label{cor:firstexample:geq}%  ii) \ 
  If $u_0$ satisfies
  \bas
	u_0(x)\leq C_0 e^{-\alpha |x|^\beta}
  \eas
  for some $C_0>0$, $\alpha>0$ and $\beta>0$, then for any $\eps>0$ we can find $t_0>0$ and $c>0$ such that
  \bas
	\E(t)\geq c t^{1/(1+\frac{n+2}{\beta}+\eps)}
	\qquad \mbox{for all } t>t_0. 
  \eas
\end{enumerate}
\end{cor}
But also quite tiny deviations from linear growth of $\E$ can be enforced by appropriate choices of the initial
data. This is indicated by a second consequence of Theorem \ref{theo300} on a corresponding logarithmic
correction for very rapidly decaying data.
\begin{cor}\label{cor:secondexample}
  Let $u_0\in C^0(\R^n)$ be a positive function,
  satisfying (\ref{10.1})-(\ref{10.3}) as well as
  \bas
	u_0(x)\leq C_0 e^{-\alpha e^{|x|^{\beta}}}
	\qquad \mbox{for all } x\in\R^n
  \eas
  with some positive constants $C_0>0$, $\alpha>0$ and $\beta>0$.
  Then for any $\eps>0$ there exist $t_0>0$ and $c>0$ such that with $u$ and $\E$ taken from
  Theorem \ref{theo10} and (\ref{def_E}), we have
  \bas
	\E(t)\geq ct\left(\ln t\right)^{-(\frac{n+2}{\beta}+\eps)} \qquad \mbox{for all } t>t_0.
  \eas
\end{cor}
{\bf Plan of the paper.} \quad
A key for our analysis consists in the observation that if for a supposedly given classical solution of (\ref{0})
we let $L(t):=\ir |\nabla u(\cdot,t)|^2$, $t>0$, then upon suitable choices of the functions $f$ and $H$ on $[0,\infty)$,
the substitution $v(s,t):=f(s)u(x,t)$ with $t=H(s)$ transforms the equation $u_t=u\Delta u + L(t)u$ satisfied by $u$
to the yet degenerate but local unforced diffusion equation 
\be{v}
	v_s=v\Delta v.
\ee
Accordingly, as a prerequisite for our existence proof we will first make sure that under essentially the same
assumptions on the initial data as in Theorem \ref{theo10}, the Cauchy problem for the latter equation possesses
a smooth positive solution enjoying certain favorable spatial decay properties which inter alia allow for 
a precise control of the respective mass evolution as well as for the important conclusion
that the corresponding Dirichlet integral depends continuously on the time variable (see Section \ref{sect_v} and
especially Lemma \ref{lem5} and Lemma \ref{lem8}).
Relying on a known fundamental asymptotic property of these solutions
(Lemma \ref{FWProp1_3}), in Section \ref{sect3}
it will thereupon be possible to construct a suitable substitution which indeed transforms
solutions of (\ref{v}) into classical solutions of (\ref{0}) with conserved unit mass and some additional
regularity properties (Lemma \ref{lem9}), the latter enabling us to justify the monotonicty property
suggested by (\ref{energy}) and thereby establish Proposition \ref{prop103} in Section \ref{sect4}.
On the basis of known refined temporal decay estimates for solutions to (\ref{v}), in Section \ref{sect5}
we will obtain Theorem \ref{theo200} as an immediate consequence of the literature, whereas
in Section \ref{sect6} a more detailed analysis will be performed so as to finally yield
the inequalities claimed in Theorem \ref{theo300}, Corollary \ref{cor:firstexample} and Corollary \ref{cor:secondexample}.
\mysection{Analysis of the Cauchy problem for $v_s=v\Delta v$}\label{sect_v}
In this section we consider the corresponding initial-value problem
\be{0v}
	\left\{ \begin{array}{ll}
	v_s=v \Delta v, \qquad & x\in\R^n, \ s>0, \\[1mm]
	v(x,0)=v_0(x), & x\in\R^n,
	\end{array} \right.
\ee
with given positive data $v_0\in L^\infty(\R^n) \cap C^0(\R^n)$, and seek for a solution to the original problem
by an adequate change of variables.
This will be achieved in Lemma \ref{lem9} on the basis of the following
result on classical solvability of (\ref{0v}) under these mild hypotheses on the initial data, as obtained
in \cite[Proposition 1.1]{fast_growth1}.
\begin{lem}\label{lem0}
  Suppose that $v_0\in L^\infty(\R^n) \cap C^0(\R^n)$ is positive.
  Then the problem (\ref{0v}) possesses at least one global classical solution $v\in C^0(\R^n\times [0,\infty))
  \cap C^{2,1}(\R^n\times (0,\infty))$ which satisfies 
  \bas
	0<v(x,s)\le \|v_0\|_{L^\infty(\R^n)}	
	\qquad \mbox{ for all $x\in\R^n$ and } s\ge 0.
  \eas
  Moreover, this solution is minimal in the sense that whenever $T\in (0,\infty]$ and
  $\tv \in C^0(\R^n\times [0,T)) \cap C^{2,1}(\R^n\times (0,T))$ are such that $\tv$ is positive and
  solves (\ref{0v}) classically
  in $\R^n\times (0,T)$, we necessarily have $v\le \tv$ in $\R^n\times (0,T)$.\\
\end{lem}
A convenient additional property of the minimal solution gained above is that it can be approximated
by classical solutions of certain regularized problems in bounded domains.
To make this more precise, throughout the sequel abbreviating $B_R:=B_R(0)$, 
let us fix a family $(v_{0R})_{R>0} \subset C^3(\bar B_R)$ of functions satisfying 
$0<v_{0R}<v_0$ in $B_R$ and $v_{0R}=0$ on $\partial B_R$ as well as
$v_{0R} \nearrow v_0$ in $\R^n$ as $R\nearrow \infty$, and consider the Dirichlet problems
\be{0vR}
	\left\{ \begin{array}{ll}
	v_{Rs}=v_R \Delta v_R, \qquad & x\in B_R, \ s>0, \\[1mm]
	v_R(x,s)=0, \qquad & x\in\partial B_R, \ s>0, \\[1mm]
	v_R(x,0)=v_{0R}(x), \qquad & x\in B_R,
 	\end{array} \right.
\ee
for $R>0$, along with the non-degenerate approximate versions thereof given by
\be{0vReps}
	\left\{ \begin{array}{ll}
	v_{R\eps s}=v_{R \eps} \Delta v_{R\eps}, \qquad & x\in B_R, \ s>0, \\[1mm]
	v_{R\eps}(x,t)=\eps, \qquad & x\in\partial B_R, \ s>0, \\[1mm]
	v_{R\eps}(x,0)=v_{0R\eps}(x), \qquad & x\in B_R,
 	\end{array} \right.
\ee
for $\eps\in (0,1)$,
where we have set
\bas
	v_{0R\eps}:=v_{0R}+\eps.
\eas
We then indeed have the following (\cite[Lemma 2.1]{fast_growth1}):
\begin{lem}\label{lem01}
  Suppose that $v_0\in L^\infty(\R^n) \cap C^0(\R^n)$ is positive.
  Then with $v_{0R}$ and $v_{0\eps R}$ as above, any of the problems (\ref{0vReps}) possesses a global classical
  solution $v_{R\eps} \in C^0(\bar B_R \times [0,\infty)) \cap C^{2,1}(\bar B_R\times (0,\infty))$. 
  As $\eps\searrow 0$, these solutions satisfy $v_{R\eps}\searrow v_R$ with a positive classical solution 
  $v_R\in C^0(\bar B_R \times [0,\infty)) \cap C^{2,1}(B_R\times (0,\infty))$ 
  of (\ref{0vR}), and moreover we have
  \bas
	v_R \nearrow v
	\quad \mbox{in } \R^n\times (0,\infty)
	\qquad \mbox{as } R\nearrow \infty,
  \eas
  where $v$ denotes the minimal classical solution of (\ref{0v}) addressed in Lemma \ref{lem0}
\end{lem}
For later reference, let us 
note that in the special case when $v_0$ is radially symmetric around the origin and nonincreasing with respect
to $|x|$, we may assume that $v_{0R}$ has the same properties. According to a standard argument based on
the comparison principle applied to (\ref{0vReps}), this entails that in this situation also 
$v_{R\eps}(\cdot,s)$ and $v_R(\cdot,s)$ are radially symmetric and nonincreasing with respect to $|x| \in [0,R]$
for each $s>0$.
\subsection{Temporally uniform spatial decay}
The goal of this section is to make sure that the minimal solution of (\ref{0v}) will inherit an initially present
spatial decay in the following temporally uniform sense:
\begin{lem}\label{lem1}
  Suppose that $v_0\in C^0(\R^n)$ is positive and such that
  \be{1.1}
	v_0(x) \to 0
	\qquad \mbox{as } |x|\to\infty.
  \ee
  Then the minimal solution $v$ of (\ref{0v}) satisfies $v(\cdot,s) \to 0$ as $|x|\to\infty$, uniformly with respect to
  $s\ge 0$; that is, we have
  \be{1.2}
	\|v\|_{L^\infty((\R^n\setminus B_R) \times (0,\infty))} \to 0
	\qquad \mbox{as } R\to\infty.
  \ee
\end{lem}
In order to prepare a proof of this, let us recall the following implication of (\ref{1.1}) on the large time behavior
of $v$, as derived in \cite[Thm. 3.1]{win_cauchy}.
\begin{lem}\label{lem03}
  Assume that the positive function $v_0\in C^0(\R^n)$ satisfies
  $v_0(x) \to 0$ as $|x|\to\infty$.
  Then the minimal solution $v$ of (\ref{0v}) has the property that
  \bas
	\|v(\cdot,s)\|_{L^\infty(\R^n)} \to 0
	\qquad \mbox{as } s\to\infty.
  \eas
\end{lem}
We will rely on this lemma in the following\\
\proofc of Lemma \ref{lem1}. \quad
  For $x\in\R^n$ replacing $v_0(x)$ with $\ov_0(x):=\|v_0\|_{L^\infty(\R^n\setminus B_{|x|})}$ if necessary, 
  on recalling the comment following Lemma \ref{lem01} we may assume that $v_0$ is radially
  symmetric and nonincreasing with respect to $|x|\ge 0$, and that for each $R>0$ also $v_{0R}$ and
  $v_R(\cdot,s)$ are radial and nonincreasing in $|x|\in [0,R]$ for $s>0$.\\
%Should we mention that we apply the comparison theorem here?
%
  Now if (\ref{1.2}) was false, there would exist $\delta_0>0$, $(R_k)_{k\in\N} \subset (0,\infty)$ and
  $(s_k)_{k\in\N} \subset (0,\infty)$ such that $R_k\to\infty$ as $k\to\infty$ and
  \bas
	\|v(\cdot,s_k)\|_{L^\infty(\R^n \setminus B_{R_k})} \ge \delta_0
	\qquad \mbox{for all } k\in\N.
  \eas
  Here in the case when $s_k\to\infty$ as $k\to\infty$, invoking Lemma \ref{lem03} we would see that
  \bas
	\delta_0 \le \|v(\cdot,s_k)\|_{L^\infty(\R^n \setminus B_{R_k})}
	\le \|v(\cdot,s_k)\|_{L^\infty(\R^n)} 
	\to 0
	\qquad \mbox{as } k\to\infty,
  \eas
  which is absurd.
  We are thus left with the case when $(s_k)_{k\in\N}$ has a finite point of accumulation, in which on passing to a
  subsequence we may assume that there exists $s_2\in [0,\infty)$ such that as $k\to\infty$ we have $s_k\to s_2$.
  We first claim that then
  \be{1.3}
	v(\cdot,s_2) \ge \delta_0
	\qquad \mbox{in } \R^n.
  \ee
  In fact, given any $x_0\in\R^n$ we may use that $R_k\to\infty$ as $k\to\infty$ to fix $k_0\in\N$ such that
  $R_k\ge |x_0|$ for all $k\ge k_0$. As $v(\cdot,s_k)$ is nonincreasing with respect to $|x|\ge 0$,
  we thus obtain that
  \bas
	v(x_0,s_k) \ge \|v(\cdot,s_k)\|_{L^\infty(\R^n \setminus B_{R_k})} \ge \delta_0
	\qquad \mbox{for all } k\ge k_0,
  \eas
  and that since by continuity of $v$ we have $v(x_0,s_k)\to v(x_0,s_2)$ as $k\to\infty$, therefore indeed
  $v(x_0,s_2)\ge \delta_0$.\abs
  Having thereby verified (\ref{1.3}), we proceed to show that actually
  \be{1.4}
	v\ge \delta_0
	\qquad \mbox{in } \R^n \times (s_2,\infty).
  \ee
  To see this, for fixed $x_0\in\R^n$ and $R>0$ we let $\varphi_R$ and $\lambda_R$ denote the first Dirichlet
  eigenfunction of $-\Delta$ in $B_R(x_0)$, normalized such that $\max_{x\in\bar B_R(x_0)} \varphi_R(x)=\varphi_R(x_0)=1$,
  and the associated principal eigenvalue $\lambda_R>0$.
  For $\delta\in (0,\delta_0)$, we then define $y_R$ to be the solution of the initial-value problem
  \be{1.5}
	\left\{ \begin{array}{l}
	y_R'(s)=-\lambda_R y_R^2(s), \qquad s>s_2, \\[1mm]
	y_R(s_2)=\delta,
	\end{array} \right.
  \ee
  that is, we let
  \bas
	y_R(s):=\frac{\delta}{1+\lambda_R \delta (s-s_2)}, \qquad s\ge s_2,
  \eas
  and thereafter we introduce a comparison function $\uv$ by writing
  \bas
	\uv(x,s):=y_R(s) \varphi_R(x),
	\qquad x\in\bar B_R(x_0), \ s\ge s_2.
  \eas
  Then clearly $\uv(x,s)=0<v(x,s)$ for all $x\in\partial B_R(x_0)$ and $s\ge s_2$, whereas (\ref{1.3}) ensures that
  also $\uv(x,s_2)\le y_R(s_2)=\delta < v(x,s_2)$ for all $x\in\bar B_R(x_0)$.
  As furthermore the definition of $\varphi_R$ and (\ref{1.5}) guarantee that
  \bas
	\uv_s - \uv \Delta \uv
	&=& y_R' \varphi_R - y_R^2 \varphi_R \Delta \varphi_R \\
	&=& y_R' \varphi_R + \lambda_R y_R^2 \varphi_R^2 \\
	&\le& y_R' \varphi_R + \lambda_R y_R^2 \varphi_R \\[1mm]%\varphi^2 \le \varphi, as \varphi\le 1 by def.
	&=& 0
	\qquad \mbox{in } B_R(x_0) \times (s_2,\infty),
  \eas
  an appropriate comparison principle (\cite{wiegner}) shows that $v\ge \uv$ in $\bar B_R(x_0)\times [s_2,\infty)$, 
  whence in particular
  \bas
	v(x_0,s) \ge \uv(x_0,s) = y_R(s) \varphi_R(x_0)
	= \frac{\delta}{1+\lambda_R \delta (s-s_2)}
	\qquad \mbox{for all } s>s_2.
  \eas
  Taking $\delta\nearrow \delta_0$ and using that $\lambda_R\to 0$ as $R\to\infty$, we thereby readily obtain that
  $v(x_0,s) \ge \delta_0$ for all $s> s_2$. This establishes (\ref{1.4}) and therefore evidently contradicts the
  outcome of Lemma \ref{lem03} also in this case.
\qed
\subsection{Time evolution of $L^p$ seminorms}
For arbitrary $p>0$,
let us next provide a rigorous counterpart of the identity
\bas
	\frac{d}{ds} \ir v^p = - p^2 \ir v^{p-1} |\nabla v|^2,
\eas
as formally obtained on testing (\ref{0v}) by $v^{p-1}$, which we state in a form that avoids time derivatives
which due to the lack of appropriate regularity assumptions on $v$ might not exist.
The correspondingly obtained result (\ref{5.1}) will frequently be applied in the sequel, inter alia 
to the value $p:=1$ for establishing 
the desired mass conservation property of our solution $u$ to the original problem (\ref{0}),
but also e.g.~to
$p:=2$ (Lemmata \ref{lem7}, \ref{lem9}) and to certain $p\in (0,1)$
(Lemma \ref{lem8}).
\begin{lem}\label{lem5}
  Assume that with some $p>0$, the positive function $v_0$ belongs to $L^p(\R^n) \cap L^\infty(\R^n) \cap C^0(\R^n)$,
 and suppose that $v$ is a bounded positive classical solution of (\ref{0v}) in $\R^n\times (0,T)$ for some 
  $T \in (0,\infty]$.
  Then $v(\cdot,s) \in L^p(\R^n)$ for all $s\in (0,T)$ and $v^{p-1} |\nabla v|^2 \in L^1(\R^n\times (0,T))$,
  and moreover we have the identity
  \be{5.1}
	\ir v^p(\cdot,s)
	+ p^2 \int_0^s \ir v^{p-1} |\nabla v|^2
	= \ir v_0^p
	\qquad \mbox{for all } s\in (0,T).
  \ee
\end{lem}
\begin{proof}
%\proof
  We first claim that
  \be{5.2}
	\ir v^p(\cdot,s)
	+ p^2 \int_0^s \ir v^{p-1} |\nabla v|^2
	\le \ir v_0^p
	\qquad \mbox{for all } s\in (0,T).
  \ee
  To establish this inequality, given $\alpha>0$ and $R>0$ we introduce $\varphi_\alpha(x):=e^{-\alpha |x|}$
  for $x\in\R^n$ as well as
  \bas
	\zeta_R(x):=\left\{ \begin{array}{ll}
	1 \qquad & \mbox{if } x\in B_R, \\[1mm]
	R+1-|x| \qquad & \mbox{if } x\in B_{R+1}\setminus B_R, \\[1mm]
	0 \qquad & \mbox{if } x\in \R^n\setminus B_{R+1}, 
	\end{array} \right.
  \eas
  and note that then both $\varphi_\alpha$ and $\zeta_R$ belong to $W^{1,\infty}(\R^n)$.
  By positivity of $v$, we may thus test (\ref{0v}) against $\zeta_R^2 \varphi_\alpha v^{p-1}$ to see that since
  $\zeta_R$ has compact support in $\R^n$, we have
  \be{5.3}
	\frac{d}{ds} \ir \zeta_R^2 \varphi_\alpha v^p
	+ p^2 \ir \zeta_R^2 \varphi_\alpha v^{p-1} |\nabla v|^2
	= - p\ir v^p \nabla v \cdot \nabla (\zeta_R^2 \varphi_\alpha)
	\qquad \mbox{for all } s\in (0,T).
  \ee
  For fixed $\eta\in (0,1)$, by means of Young's inequality we find that herein
  \be{5.33}
	\bigg| - p\ir v^p\nabla v \cdot \nabla (\zeta_R^2 \varphi_\alpha) \bigg|
	\le \eta p^2 \ir \zeta_R^2 \varphi_\alpha v^{p-1} |\nabla v|^2
	+ \frac{1}{4\eta} \ir \frac{|\nabla (\zeta_R^2 \varphi_\alpha)|^2}{\zeta_R^2 \varphi_\alpha} \cdot v^{p+1}
	\qquad \mbox{for all } s\in (0,T),
  \ee
  where we can estimate
  \bea{5.34}
	\Big| \nabla (\zeta_R^2 \varphi_\alpha) \Big|^2
	&=& \Big| 2\zeta_R\varphi_\alpha \nabla \zeta_R + \zeta_R^2 \nabla\varphi_\alpha \Big|^2 \nn\\
	&\le& 8\zeta_R^2 \varphi_\alpha^2 |\nabla \zeta_R|^2 + 2\zeta_R^4 |\nabla\varphi_\alpha|^2
	\qquad \mbox{a.e.~in } \R^n
  \eea
  and hence
  \be{5.4}
	\frac{1}{4\eta}\ir \frac{|\nabla (\zeta_R^2 \varphi_\alpha)|^2}{\zeta_R^2 \varphi_\alpha} \cdot v^{p+1}
	\le \frac{2}{\eta} \io |\nabla \zeta_R|^2 \varphi_\alpha v^{p+1}
	+ \frac{1}{2\eta} \io \zeta_R^2 \cdot \frac{|\nabla\varphi_\alpha|^2}{\varphi_\alpha} \cdot v^{p+1}
	\qquad \mbox{for all } s\in (0,T).
  \ee
  Recalling that $M:=\|v\|_{L^\infty(\R^n\times (0,T))}$ is finite by assumption, since $|\nabla \zeta_R|\le 1$
  a.e.~in $\R^n$ and $\supp \nabla\zeta_R \subset \bar B_{R+1}\setminus B_R$ we see that
  \be{5.5}
	\frac{2}{\eta} \io |\nabla \zeta_R|^2 \varphi_\alpha v^{p+1}
	\le \delta_{\eta\alpha R} := \frac{2M^{p+1}}{\eta} \int_{B_{R+1}\setminus B_R} \varphi_\alpha
	\qquad \mbox{for all } s\in (0,T),
  \ee
  where evidently for each fixed $\eta \in (0,1)$ and $\alpha>0$ we have
  \be{5.55}
	\delta_{\eta\alpha R} \to 0
	\qquad \mbox{as } R\to\infty.
  \ee
  In estimating the second summand on the right of (\ref{5.4}) we use that 
  $\frac{|\nabla\varphi_\alpha|^2}{\varphi_\alpha^2}=\alpha^2$ in $\R^n\setminus \{0\}$ to obtain
  \be{5.56}
	\frac{1}{2\eta} \io \zeta_R^2 \cdot \frac{|\nabla\varphi_\alpha|^2}{\varphi_\alpha} \cdot v^{p+1}
	\le \frac{M\alpha^2}{2\eta} \ir \zeta_R^2 \phii_\alpha v^p
	\qquad \mbox{for all } s\in (0,T).
  \ee
  Combined with (\ref{5.3})-(\ref{5.5}), this shows that for
  $y_{\alpha R}(s):=\ir \zeta_R^2 \varphi_\alpha v^p(\cdot,s), \ s\in [0,T)$,
  and $f_{\alpha R}(s):= p^2 \ir \zeta_R^2 \varphi_\alpha v^{p-1}(\cdot,s)|\nabla v(\cdot,s)|^2, \ s\in (0,T)$, writing $c_1:=\frac{M}{2\eta}$ we have
  \be{5.6}
	y_{\alpha R}'(s) + (1-\eta) f_{\alpha R}(s) 
	\le c_1 \alpha^2 y_{\alpha R}(s) + \delta_{\eta \alpha R}
	\qquad \mbox{for all } s\in (0,T),
  \ee
  which on integration yields
  \bea{5.7}
	e^{-c_1 \alpha^2 s} y_{\alpha R}(s)
	+ (1-\eta) \int_0^s e^{-c_1 \alpha^2 \sigma} f_{\alpha R}(\sigma) d\sigma
	&\le& y_{\alpha R}(0)
	+ \delta_{\eta \alpha R} \int_0^s e^{-c_1 \alpha^2 \sigma} d\sigma \nn\\
	&\le& y_{\alpha R}(0)
	+ \frac{\delta_{\eta \alpha R}}{c_1 \alpha^2}
	\qquad \mbox{for all } s\in (0,T).
  \eea
  We now observe that since $\zeta_R \nearrow 1$ in $\R^n$ as $R\nearrow \infty$, Beppo Levi's theorem asserts that
  for each $s\in (0,T)$,
  \be{5.77}
	y_{\alpha R}(s) \nearrow y_\alpha(s):=\ir \varphi_\alpha v^p(\cdot,s)
	\qquad \mbox{as } R\nearrow \infty,
  \ee
  and that for all $s\in (0,T)$ we have
  \be{5.78}
	f_{\alpha R}(s) \nearrow f_\alpha(s):= p^2 \ir \varphi_\alpha v^{p-1}(\cdot,s) |\nabla v(\cdot,s)|^2
	\qquad \mbox{as } R\nearrow \infty.
  \ee
  According to (\ref{5.55}), from (\ref{5.7}) we thus infer that
  \be{5.8}
	e^{-c_1 \alpha^2 s} y_\alpha(s)
	+ (1-\eta) \int_0^s e^{-c_1 \alpha^2 \sigma} f_\alpha(\sigma) d\sigma
	\le y_\alpha(0)
	\qquad \mbox{for all } s\in (0,T).
  \ee
  As a similar argument based on monotone convergence warrants that for $s\in [0,T)$ we have
  \be{5.88}
	y_\alpha(s)\nearrow y(s):= \ir v^p(\cdot,s)
	\qquad \mbox{as } \alpha\searrow 0,
  \ee
  and that
  \be{5.89}
	f_\alpha(s)\nearrow f(s):= p^2 \ir v^{p-1}(\cdot,s)|\nabla v(\cdot,s)|^2
	\qquad \mbox{as } \alpha\searrow 0
  \ee
  for each $s\in (0,T)$, once more by Beppo Levi's theorem we conclude from (\ref{5.8}) that
  \be{5.9}
	y(s) + (1-\eta) \int_0^s f(\sigma)d\sigma \le y(0)
	\qquad \mbox{for all } s\in (0,T).
  \ee
  which proves (\ref{5.2}), because $\eta\in (0,1)$ was arbitrary.\abs
  In order to show that conversely also
  \be{5.99}
	\ir v^p(\cdot,s)
	+ p^2 \int_0^s \ir v^{p-1} |\nabla v|^2
	\ge \ir v_0^p
	\qquad \mbox{for all } s\in (0,T),
  \ee
  given $\eta\in (0,1)$ we again combine (\ref{5.3}) with (\ref{5.33})-(\ref{5.5}) and (\ref{5.56}) to see that
  \bas
	y_{\alpha R}'(s) + (1+\eta) f_{\alpha R}(s) 
	\ge - c_1 \alpha^2 y_{\alpha R}(s) - \delta_{\eta \alpha R}
	\qquad \mbox{for all } s\in (0,T),
  \eas
  and that hence
  \bas
	e^{c_1 \alpha^2 s} y_{\alpha R}(s)
	+ (1+\eta) \int_0^s e^{c_1 \alpha^2 \sigma} f_{\alpha R}(\sigma) d\sigma
	\ge y_{\alpha R}(0)
	- \frac{\delta_{\eta \alpha R}}{c_1 \alpha^2}
	\qquad \mbox{for all } s\in (0,T).
  \eas
  On taking $R\searrow \infty$, by (\ref{5.55}), (\ref{5.77}) and (\ref{5.78}) we thereby obtain that
  \be{5.10}
	e^{c_1 \alpha^2 s} y_\alpha(s)
	+ (1+\eta) \int_0^s e^{+c_1 \alpha^2 \sigma} f_\alpha(\sigma) d\sigma
	\ge y_\alpha(0)
	\qquad \mbox{for all } s\in (0,T),
  \ee
  where in passing to the limit in the integral involving $f_{\alpha R}$ we have used the dominated convergence theorem
  along with the inequality
  \bas
	e^{c_1 \alpha^2 \sigma} f_{\alpha R}(\sigma) \le e^{c_1 \alpha^2 s} f_\alpha(\sigma),
  \eas
  valid for all $\sigma\in (0,s)$ by (\ref{5.78}), and the fact that $f_\alpha$ is integrable over $(0,s)$ due to 
  e.g.~(\ref{5.9}) and (\ref{5.89}), for any $s\in(0,T)$.
  Likewise, one more application of the dominated convergence theorem on the basis of (\ref{5.88}) and (\ref{5.89})
  finally shows that (\ref{5.10}) entails the inequality
  \bas
	y(s) + (1+\eta) \int_0^s f(\sigma)d\sigma \ge y(0)
	\qquad \mbox{for all } s\in (0,T),
  \eas
  and thereby yields (\ref{5.99}) on letting $\eta\searrow 0$.
\end{proof}
\subsection{Time continuity and integrability of $\ir |\nabla v|^2$}
In view of the regularity requirements in Definition \ref{defi11}, an inevitable task consists in showing
continuous dependence of the Dirichlet integral in (\ref{0}) on the time variable. In the transformed problem,
this amounts to showing that $0\le s \mapsto \ir |\nabla v(\cdot,s)|^2$ is continuous, which is prepared by the following 
statement on a corresponding temporally uniform spatial decay property
that will beyond this also be useful in Lemma \ref{lem9} and Proposition \ref{prop103} below.
\begin{lem}\label{lem7}
  Suppose that $v_0\in L^2(\R^n)\cap L^\infty(\R^n)\cap C^0(\R^n)$ is positive, and that with some $s_2>0$ and 
  $T\in (s_2,\infty]$, $v$ is a bounded positive classical solution of (\ref{0v}) in $\R^n \times (0,T)$ which is such that
  $v(\cdot,s_2)$ belongs to $W^{1,2}(\R^n)$. Then $v(\cdot,s) \in W^{1,2}(\R^n)$ for all $s\in (s_2,T)$ and
  $v|\Delta v|^2 \in L^1(\R^n\times (s_2,T))$, and for
  each $R>0$ we have
  \begin{align}\label{7.1}
	\int_{\R^n \setminus B_{R+1}} |\nabla v(\cdot,s)|^2
	+ \int_{s_2}^s \int_{\R^n \setminus B_{R+1}} v|\Delta v|^2
	\le& \int_{\R^n\setminus B_R} |\nabla v(\cdot,s_2)|^2 \nn\\
	&  + 4 \int_{s_2}^s \int_{B_{R+1}\setminus B_R} v|\nabla v|^2
	\qquad \mbox{for all } s\in (s_2,T).
  \end{align}
\end{lem}
\proof
  We fix $R>0$ and let $\zeta_\rho \in W^{1,\infty}(\R^n)$ be defined by
  \bas
	\zeta_\rho(x):=
	\left\{ \begin{array}{ll}
	0, \qquad & x\in B_R, \\[1mm]
	|x|-R, \qquad & x\in B_{R+1}\setminus B_R, \\[1mm]
	1, \qquad & x\in B_\rho\setminus B_{R+1}, \\[1mm]
	\rho+1-|x|, \qquad & x\in B_{\rho+1}\setminus B_\rho, \\[1mm]
	0, \qquad & x\in \R^n\setminus B_{\rho+1}, 
	\end{array} \right.
  \eas
  for $\rho>R+1$. Then since $v$ is smooth in $\R^n \times (0,T)$, we may use (\ref{0v}) to compute
  \bea{7.2}
	\hspace*{-8mm}
	\frac{1}{2} \ir \zeta_\rho^2 |\nabla v(\cdot,s)|^2
	- \frac{1}{2} \ir \zeta_\rho^2 |\nabla v(\cdot,s_2)|^2
	&=& - \int_{s_2}^s \ir \zeta_\rho^2 v_s\Delta v
	- 2\int_{s_2}^s \ir \zeta_\rho v_s (\nabla v \cdot \nabla \zeta_\rho) \nn\\
	&=& - \int_{s_2}^s \ir \zeta_\rho^2 v |\Delta v|^2
	- 2\int_{s_2}^s \ir \zeta_\rho v \Delta v (\nabla v \cdot \nabla \zeta_\rho)
%	\qquad \mbox{for all } s\in (s_2,T),
  \eea
  for all $s\in (s_2,T)$, where by Young's inequality,
  \bea{7.3}
	- 2\int_{s_2}^s \ir \zeta_\rho v \Delta v (\nabla v \cdot \nabla \zeta_\rho)
	&\le& \frac{1}{2} \int_{s_2}^s \ir \zeta_\rho^2 v|\Delta v|^2
	+ 2 \int_{s_2}^s \ir |\nabla \zeta_\rho|^2 v|\nabla v|^2 \nn\\
	&\le& \frac{1}{2} \int_{s_2}^s \ir \zeta_\rho^2 v|\Delta v|^2
	+ 2 \int_{s_2}^s \int_{B_{R+1}\setminus B_R} v|\nabla v|^2 \nn\\
	& & + 2 \int_{s_2}^s \int_{B_{\rho+1}\setminus B_\rho} v|\nabla v|^2 
	\qquad \mbox{for all } s\in (s_2,T),
  \eea
  because $|\nabla \zeta_\rho| \le 1$ a.e.~in $\R^n$ and $\nabla \zeta_\rho=0$ outside
  $(B_{R+1}\setminus B_R)\cup (B_{\rho+1}\setminus B_\rho)$.
  Here we note that our assumptions on $v_0$ and $v$ together with the outcome of Lemma \ref{lem5}, applied to $p:=2$,
  ensure that $v|\nabla v|^2 \in L^1(\R^n\times (0,T))$ and hence
  \bas
	\int_{s_2}^s \int_{B_{\rho+1}\setminus B_\rho} v|\nabla v|^2 \to 0
	\qquad \mbox{as } \rho\to\infty
  \eas
  by the dominated convergence theorem. 
  Since as $\rho\to\infty$ we furthermore have
  \bas
	\zeta_\rho(x) \nearrow \zeta(x):=
	\left\{ \begin{array}{ll}
	0, \qquad & x\in B_R, \\[1mm]
	|x|-R, \qquad & x\in B_{R+1}\setminus B_R, \\[1mm]
	1, \qquad & x\in \R^n \setminus B_{R+1},
	\end{array} \right.
  \eas
  from (\ref{7.2}) and (\ref{7.3}) we infer on an application of Beppo Levi's theorem that
  \bas
	\frac{1}{2} \ir \zeta^2 |\nabla v(\cdot,s)|^2
	+ \frac{1}{2} \int_{s_2}^2 \ir \zeta^2 v|\Delta v|^2
	\le \frac{1}{2} \ir \zeta^2 |\nabla v(\cdot,s_2)|^2
	+ 2 \int_{s_2}^s \int_{B_{R+1}\setminus B_R} v|\nabla v|^2
%	\qquad \mbox{for all } s\in (s_2,T),
  \eas
  for all $s\in (s_2,T)$,
  which in view of the evident inequalities $\chi_{\R^n\setminus B_{R+1}} \le \zeta \le \chi_{\R^n \setminus B_R}$
  immediately yields both the claimed regularity properties and (\ref{7.1}).
\qed

We can now establish the desired continuity property.
\begin{lem}\label{lem8}
  Suppose that for some $p\in (0,1)$, the positive function $v_0$ belongs to 
  $L^p(\R^n)\cap C^0(\R^n)$ and satisfies
  \be{8.1}
	v_0(x)\to 0
	\qquad \mbox{as } |x|\to\infty.
  \ee
  Then the minimal solution $v$ of (\ref{0v}) has the property that
  $v(\cdot,s)\in W^{1,2}(\R^n)$ for all $s>0$, that $\na v\in L^2(\R^n\times(0,\infty))$, and that
  \be{8.2}
	(0,\infty) \ni s \mapsto \ir |\nabla v(\cdot,s)|^2 
	\quad \mbox{is continuous.}
  \ee
\end{lem}
\proof
  As our assumptions necessarily require that $v_0\in L^1(\R^n)$, %\red{weil $v_0\to 0$ auch $v_0\in L^\infty$ nach sich zieht und man dann interpolieren kann.} 
 a first application of Lemma \ref{lem5} shows that
  indeed $\na v\in L^2(\R^n\times(0,\infty))$, which due to the fact that also $v_0\in L^2(\R^n)$ may be combined with 
  Lemma \ref{lem7} to imply that $v(\cdot,s)$ belongs to $W^{1,2}(\R^n)$ actually for all $s>0$.
  To verify the continuity property (\ref{8.2}), we fix $\sst>0$ and $\eps>0$, and note that since $v_0\in L^p(\R^n)$,
  once more employing Lemma \ref{lem5} we see that
  \bas
	I:= \int_0^\infty \ir v^{p-1} |\nabla v|^2
  \eas
  is finite. In particular, this entails that it is possible to find $s_2>0$ such that $s_2\in (\sst-1,\sst)$ and
  \be{8.3}
	\ir v^{p-1}(\cdot,s_2) |\nabla v(\cdot,s_2)|^2
	\le c_1:= \frac{I}{\min\{\frac{\sst}{2},1\}},
  \ee
  for otherwise we could draw the absurd conclusion that
  \bas
	I \ge \int_{\max\{\frac{\sst}{2},\sst-1\}}^{\sst} \ir v^{p-1} |\nabla v|^2 
	> c_1 \cdot \Big(\sst - \max \Big\{ \frac{\sst}{2}, \sst-1\Big\} \Big)
	= c_1 \cdot \min \Big\{ \frac{\sst}{2}, 1 \Big\} =I.
  \eas
  We next use that $p<1$ in choosing $\eta>0$ small enough such that
  \be{8.4}
	c_1 \eta^{1-p} < \frac{\eps}{6}
	\qquad \mbox{and} \qquad
	4\eta^{2-p} I < \frac{\eps}{6},
  \ee
  and thereupon rely on the uniform spatial decay property of $v$, as asserted by Lemma \ref{lem1}, to find some suitably
  large $R>0$ fulfilling
  \be{8.5}
	v(x,s)< \eta
	\qquad \mbox{for all $x\in \R^n \setminus B_R$ and all } s>0.
  \ee
  Finally, writing
  \be{8.55}
	\ssst:=\max \Big\{ \frac{\sst}{2}, s_2\Big\} \in (\sst-1,\sst),
  \ee
  we know from the fact that $v$ is a classical solution %standard parabolic regularity theory (\cite{LSU}) \red{Erstens: Das ist ein ganz h\"asslicher Verweis (weil furchtbar ungenau). Zweitens: Hast du eine spezielle Aussage im Sinn? Ich finde jedenfalls auch mit einigem Suchen kein Theorem, das direkt anwendbar ist; welches habe ich \"ubersehen? Drittens: Wozu \"uberhaupt? Nach Lemma \ref{lem0} (oder besser: Def. \ref{defi11}) ist doch $v\in C^{2,1}$, also $\na v$ stetig.} 
  that $\nabla v$ is continuous, and hence uniformly
  continuous, in $\bar B_{R+1} \times [\ssst,\sst+1]$, whence in particular we can pick 
  $\delta\in (0,\min \{1,\sst-\ssst\})$ such that
  \be{8.6}
	\bigg| \int_{B_{R+1}} |\nabla v(\cdot,s)|^2 - \int_{B_{R+1}} |\nabla v(\cdot,\sst)|^2 \bigg|
	< \frac{\eps}{3}
	\qquad \mbox{for all } s\in (\sst-\delta,\sst+\delta).
  \ee
  Now for $s>0$, we split
  \be{8.7}
	\ir |\nabla v(\cdot,s)|^2 - \ir |\nabla v(\cdot,\sst)|^2
	= I_1(s)+I_2(s)+I_3,
  \ee
  where
  \bas
	I_1(s):=\int_{B_{R+1}} |\nabla v(\cdot,s)|^2 - \int_{B_{R+1}} |\nabla v(\cdot,\sst)|^2
  \eas
  as well as
  \bas
	I_2(s):=\int_{\R^n\setminus B_{R+1}} |\nabla v(\cdot,s)|^2
	\quad \mbox{and} \quad
	I_3:=\int_{\R^n\setminus B_{R+1}} |\nabla v(\cdot,\sst)|^2,
  \eas
  and where according to (\ref{8.6}) we already know that
  \be{8.8}
	|I_1(s)| < \frac{\eps}{3}
	\qquad \mbox{for all } s\in (\sst-\delta,\sst+\delta).
  \ee
  In order to estimate the two remaining terms on the right of (\ref{8.7}), we first note that as a consequence of
  \eqref{8.5}, \eqref{8.3} and the left inequality in (\ref{8.4}), 
  \bas
	\int_{\R^n\setminus B_R} |\nabla v(\cdot,s_2)|^2
	&=& \int_{\R^n\setminus B_R} v^{1-p}(\cdot,s_2) \cdot v^{p-1}(\cdot,s_2)|\nabla v(\cdot,s_2)|^2 \\
	&\le& \eta^{1-p} \int_{\R^n\setminus B_R} v^{p-1}(\cdot,s_2)|\nabla v(\cdot,s_2)|^2 \\
	&\le& \eta^{1-p} c_1 \\
	&<& \frac{\eps}{6}.
  \eas
  Therefore, Lemma \ref{lem7} guarantees that
  \bas
	\int_{\R^n\setminus B_{R+1}}  |\nabla v(\cdot,s)|^2
	< \frac{\eps}{6} + 4\int_{s_2}^s \int_{B_{R+1}\setminus B_R} v|\nabla v|^2 
	\qquad \mbox{for all } s>s_2,
  \eas
  where again by \eqref{8.5}, in view of the definition of $I$, and by the second restriction in (\ref{8.4}) we see that 
  \bas
	4\int_{s_2}^s \int_{B_{R+1}\setminus B_R} v|\nabla v|^2
	&\le& 4\int_0^\infty \int_{\R^n \setminus B_R} v|\nabla v|^2 \\
	&\le& 4\eta^{2-p} \int_0^\infty \int_{\R^n\setminus B_R} v^{p-1} |\nabla v|^2 \\
	&\le& 4\eta^{2-p}I\\
	&<& \frac{\eps}{6}
	\qquad \mbox{for all } s>s_2,
  \eas
  so that in fact
  \bas
	\int_{\R^n\setminus B_{R+1}}  |\nabla v(\cdot,s)|^2
	<\frac{\eps}{3}
	\qquad \mbox{for all } s>s_2.
  \eas
  Since according to our choice of $\delta$ and our definition (\ref{8.55}) of $\ssst$ we know that
  \bas
	s_2+\delta < s_2+(\sst-\ssst) \le s_2+(\sst-s_2)=\sst,
  \eas
  this firstly shows that
  \bas
	|I_2(s)| <\frac{\eps}{3}
	\qquad \mbox{for all } s\in (\sst-\delta,\sst+\delta),
  \eas
  and thus clearly warrants that also
  \bas
	|I_3| = |I_2(\sst)| < \frac{\eps}{3}.
  \eas
  In light of (\ref{8.7}), in conjunction with (\ref{8.8}) this shows that
  \bas 
	\bigg| \ir |\nabla v(\cdot,s)|^2 - \ir |\nabla v(\cdot,\sst)|^2 \bigg|
	< \eps
	\qquad \mbox{for all } s\in (\sst-\delta,\sst+\delta)
  \eas
  and thereby completes the proof of (\ref{8.2}).
\qed
\mysection{Classical unit-mass solutions of (\ref{0}). Proof of Theorem \ref{theo10}}\label{sect3}
Having derived some properties of solutions to \eqref{0v}, we proceed to use these for the construction of solutions to 
\eqref{0} through an appropriate transformation 
involving the functions introduced in the following. 
\begin{defi}\label{def:fcts}
  Given a positive function $u_0\in C^0(\R^n)\cap L^p(\R^n)$ for some $p\in(0,1)$ which satisfies $\ir u_0=1$ an well as 
  $u_0(x)\to 0$ as $|x|\to \infty$, we let $v$ be the minimal solution of \eqref{0v} emanating from initial data $v_0:=u_0$ 
  and set  
  \begin{equation}\label{def:K}
 	K(s):=\ir |\na v(\cdot,s)|^2, \qquad s\in(0,\infty),
  \end{equation}
  and we let $H\colon [0,\infty)\to[0,\infty)$ by 
  \begin{equation}\label{def:H}
 	H(s):=s -\int_0^s\int_0^\sigma K(\tau) d\tau d\sigma, \qquad s\in [0,\infty).
  \end{equation}
  Moreover we define 
  \begin{equation}\label{def:G}
 	G(s):=\frac1{1-\int_0^s K(\tau)d\tau} ,\qquad s\in[0,\infty),
  \end{equation}
  as well as 
  \begin{equation}\label{def_g}
 	g(t):=G(h(t)), \qquad t\in[0,\infty),
  \end{equation}
  and let 
  \begin{equation}\label{def_u}
 	u(x,t):=g(t) v(x,h(t)), \qquad x\in\R^n, t\in[0,\infty),
  \end{equation}
  and 
  \begin{equation}\label{def:L}
 	L(t):=\ir |\na u(\cdot,t)|^2,\qquad t\in(0,\infty).
  \end{equation}
\end{defi}
In order to make sure that $H$ is surjective, let us recall the following from \cite[Prop. 1.3]{fast_growth1}.
\begin{lem}\label{FWProp1_3}
  Let $v_0\in C^0(\R^n)$ positive and let $v$ be a global positive classical solution of \eqref{0v}. Then for any $R>0$ 
  we have
  \bas
  	\inf_{x\in B_R} \Big\{ sv(x,s) \Big\} \to \infty \qquad \mbox{ as } s\to\infty. 
  \eas
\end{lem}
Indeed, this asserts invertibility of $H$, as contained in the next statement.
\begin{lem}\label{lem:h}
  Suppose that for some $p\in (0,1)$, 
  $u_0\in L^p(\R^n)\cap C^0(\R^n)$ is positive and satisfies $\ir u_0=1$ as well as
  $u_0(x)\to 0$ as $|x|\to\infty$.
  Then the function $H\colon [0,\infty)\to[0,\infty)$ defined in \eqref{def:H} is bijective and 
  \be{def_h}
	h:=H^{-1} \colon[0,\infty)\to[0,\infty)
  \ee
  is well-defined. Moreover, $h$ belongs to $C^2((0,\infty))\cap C^0([0,\infty))$, and for $t\in(0,\infty)$ we have 
  \begin{equation}\label{eq:hstrich}
  	h'(t)=\frac{1}{H'(h(t))} \;\;\;\mbox{ and }\;\;\;
	 h''(t)=-\frac{1}{H'(h(t)))^2}H''(h(t))h'(t) = -\frac{H''(h(t))}{H'(h(t)))^3},
  \end{equation}
  where $H'(s)>0$ for all $s\in(0,\infty)$.
\end{lem}
\proof 
  Due to the conditions on $u_0$, Lemma \ref{lem8} asserts that $K$ is continuous on $(0,\infty)$ and Lemma \ref{lem5} 
  additionally shows that $K\in L^1((0,\infty))$ and correspondingly $H\in C^2((0,\infty))\cap C^0([0,\infty))$. 
  Furthermore, again according to Lemma \ref{lem5}
  \begin{equation}\label{eq:Hstrich}
  	H'(s)=1 -\int_0^s K(\tau)d\tau = \ir v(\cdot,s)>0  \qquad \mbox{ for all }s\in(0,\infty).
  \end{equation}%positivity (last step): from Lemma \ref[lem0}. 
  Hence, $H$ is strictly monotone and thus injective. 
  Moreover, by Lemma \ref{FWProp1_3} we can find $s_0>0$ such that for all $s>s_0$ we have 
  $\inf_{x\in B_1} sv(x,s) \ge  \frac1{|B_1|}$. 
  With this, we see, again using Lemma \ref{lem5}, that 
  \begin{align*}
  	H(s)=&\int_0^s\left(\ir u_0-\int_0^\sigma K(\tau) d\tau\right)d\sigma 
	= \int_0^s\ir v(s) ds 
	\ge \int_{s_0}^s \bigg\{\int_{B_1} \sigma v(x,\sigma)  dx \bigg\} \cdot \frac{1}{\sigma} d\sigma\\
  	\ge&\int_{s_0}^s|B_1| \inf_{x\in B_1} \Big\{ \sigma v(x,\sigma) \Big\} \cdot \frac1\sigma d\sigma
	\geq\int_{s_0}^s \frac1\sigma d\sigma \to \infty
  \end{align*}
  as $s\to \infty$ and we can infer surjectivity of $H$. Altogether, we see  
  that $h\in C^0([0,\infty))$ is well-defined. 
  In addition, positivity of $H'$ also shows that \eqref{eq:hstrich} holds and $h\in C^2((0,\infty))$.
\qed
Let us list some regularity properties and further connections between the functions introduced in 
Definition \ref{def:fcts} and in (\ref{def_h}).
\begin{lem}\label{lem:reg_and_identities}
  Let $u_0\in L^p(\R^n)\cap C^0(\R^n)$ for some $p\in(0,1)$ be positive with $\ir u_0=1$  
  Then $K\in L^1((0,\infty))$, $G,g \in C^1((0,\infty))\cap C^0([0,\infty))$, 
  $L\in C^0((0,\infty))\cap L^1_{loc}([0,\infty))$, $u\in C^{2,1}(\R^n\times(0,\infty))\cap C^0(\R^n\times[0,\infty))$ and 
  \begin{align}
	\label{eq:identity_G} 
	G(s)=&\frac{1}{H'(s)}&& \mbox{ for all } s>0,\\
	\label{eq:identity_hstrichg} 
	h'(t)=&g(t)&&\mbox{ for all } t>0,\\
	\label{eq:identity_GstrichG2K} 
	G'(s)=&G^2(s)K(s)&&\mbox{ for all } s>0,\\
	\label{eq:identity_LG2K} 
	L(t)=&G^2(h(t))K(h(t))&&\mbox{ for all } t>0,\\
	\label{eq:identity_gstrichgL} 
	g'(t)=&g(t)L(t)&&\mbox{ for all } t>0.
  \end{align}
\end{lem}
\begin{proof}%\proof
  Since $H'(s)=1-\int_0^s K(\sigma)d\sigma$, obviously $G(s)=\frac1{H'(s)}$ for $s>0$ and $G\in C^1((0,\infty))$ follows 
  from Lemma \ref{lem:h}.
  Lemma \ref{lem5} warrants that $K\in L^1((0,\infty))$, admitting the conclusion that $G\in C^0([0,\infty))$. 
  The asserted regularity of $g$ thereupon becomes an immediate consequence of \eqref{def_g}. 
  Recalling that $v\in C^{2,1}(\R^n\times(0,\infty))\cap C^0(\R^n\times[0,\infty))$ and 
  $h\in C^1((0,\infty))\cap C^0([0,\infty))$, we find that 
  $u\in C^{2,1}(\R^n\times(0,\infty))\cap C^0(\R^n\times[0,\infty))$.\\
  By \eqref{eq:hstrich}, \eqref{eq:identity_G} and \eqref{def_g}, \(  h'(t)=\frac1{H'(h(t))}=G(h(t))=g(t)\) 
  for all $t\in(0,\infty)$. Straightforward application of the chain rule combined with \eqref{eq:identity_G} 
  moreover shows that 
  \bas
  	G'(s) = -\frac{1}{(H'(s))^2} H''(s) = G^2(s) K(s), \qquad s\in(0,\infty).
  \eas
  According to the definitions of $K$, $g$ and $u$, for any $t>0$ we have 
  \begin{align*}
  	G^2(h(t))K(h(t)) = g^2(t) \ir |\na v(x,h(t)|^2 dx =\ir |\na (g(t)v(x,h(t)))|^2 dx = \ir |\na u(x,t)|^2 dx= L(t)
  \end{align*}
  and from the known regularity properties of  $G$, $K$, $h$ we can infer $L\in C^0((0,\infty))\cap L^1_{loc}([0,\infty))$. 
  Differentiation of \eqref{def_g} together with 
  \eqref{eq:identity_GstrichG2K}, \eqref{eq:identity_LG2K} leads to 
  \begin{align*}
	g'(t)=G'(h(t))h'(t)=G^2(h(t))K(h(t)) h'(t)= L(t) g(t), \qquad t\in(0,\infty). %\qedhere
  \end{align*}
\end{proof}
%\qed

%
%
%
\begin{lem}\label{lem9}
  Suppose that for some $p\in (0,1)$, 
  $u_0\in L^p(\R^n)\cap C^0(\R^n)$ is positive and satisfies $\ir u_0=1$ as well as
  $u_0(x)\to 0$ as $|x|\to\infty$.
  Let $v$ denote the minimal solution of (\ref{0v}) emanating from $v_0:=u_0$, and let $u$ be defined by (\ref{def_u})
  with $g$ and $h$ given by (\ref{def_g}) and (\ref{def_h}).
  Then $u$ is a  positive classical solution of (\ref{0}) in $\R^n\times (0,\infty)$ 
  with the property that
  \be{9.1}
	\ir u(\cdot,t)=1
	\qquad \mbox{for all } t>0.
  \ee
  Moreover, for all $T>0$ we have
  \be{9.2}
	\int_0^T \ir u|\Delta u|^2 <\infty
  \ee
  and
  \be{9.3}
	\int_0^T \ir u|\nabla u|^2 < \infty.
  \ee
\end{lem}
\proof
  From Lemma \ref{lem:reg_and_identities} we know that the function $u$ satisfies 
  $u\in C^{2,1}(\R^n\times(0,\infty))\cap C^0(\R^n\times[0,\infty))$, moreover (\ref{9.2}) and (\ref{9.3}) are valid 
  according to \eqref{def_u}, Lemma \ref{lem7} and Lemma \ref{lem5}, the latter again being applicable since 
  $u_0\in L^2(\R^n)$ due to our assumptions.
  As $g(0)=G(h(0))=G(0)=1$, we have
  \[
  	u(x,0)=g(0)v(x,h(0))=v(x,0)= u_0(x),
  \]
  whereas by \eqref{0v}, \eqref{eq:identity_hstrichg}, \eqref{eq:identity_gstrichgL} and \eqref{def_u} we find that
  \bas
  	u_t(x,t)
	&=&   g(t) v_s(x,h(t))h'(t) + g'(t) v(x,h(t))\\ 
  	&=& g(t) v(x,h(t)) \Delta v(x,h(t)) g(t) + g(t)L(t) v(x,h(t))\\
  	&=& u(x,t) \Delta u(x,t) +  u(x,t) L(t)
	\qquad \mbox{for all $x\in\R^n$ and } t>0,
  \eas
  and that thus \eqref{0} holds.
  To see that the total mass remains constant,
  again letting $s:=h(t)$ we note that $u(x,t)=G(s)v(x,s)$ for all $x\in\R^n$ and $t>0$ and hence, by Lemma \ref{lem5}, indeed
  \bas
  	\ir u(x,t)=G(s) \ir v(x,s) 
	= \frac{1}{1-\int_0^s K(\sigma) d\sigma} \left[\ir u_0 - \int_0^s K(\sigma) d\sigma\right]=1 
  \eas
  for all $t>0$.
\qed
 
 %\red{Beschr\"anktheit von $u$ sehe ich hier nicht. Der Vorfaktor $G$ ist unbeschr\"ankt, das macht es schwer. (Das war anders, als wir noch mit dem $G_0\in(0,1)$ in der Definition von $G$ gearbeitet haben -- aber da passte dann die Massenerhaltung nicht...) - Die etwas feineren Absch\"atzungen f\"ur dieses Konstrukt $\frac{v}{|v|_{L^1}}$ aus \cite{fast_growth1} scheinen f\"ur gro\ss e $t$ auf $ct^{-\eps} < v(t) <Ct^{+\eps}$ zu fuehren.}
%
%
%
%
%
%
\proofc of Theorem \ref{theo10}. \quad
  All parts of the statement are immediate from Lemma \ref{lem9}.
\qed
\mysection{Decay of $\frac{\E(t)}{t}$ for arbitrary initial data}\label{sect4}
%\mysection{Unconditional decay of 	%$\ir |\nabla u|^2$ and of 
%$\frac{\E(t)}{t}$. Proof of Proposition \ref{prop103}}
%
%
%
%
%
%
%
%
In the preparatory Lemma \ref{lem101} and Lemma \ref{lem102}, we shall refer to the definition of the functions
$\zeta _R\in W^{1,\infty}(\R^n)$ which for $R>0$ we introduce by letting
\be{xi}
	\zeta _R(x):=\left\{ \begin{array}{ll}
	1 \qquad & \mbox{if } |x|\le R, \\[1mm]
	R+1-|x| \qquad & \mbox{if } R<|x|<R+1, \\[1mm]
	0 \qquad & \mbox{if } |x|\ge R+1.
	\end{array} \right.
\ee
By means of a corresponding cut-off procedure we can firstly provide a clean argument for the following functional
inequality for nonnegative (sub-)unit-mass functions enjoying a certain additional regularity property.
\begin{lem}\label{lem101}
  Let $\varphi\in C^2(\R^n)$ be nonnegative and such that
  \be{101.2}
	\ir \varphi \le 1
  \ee
  as well as
  \be{101.1}
	\ir \varphi|\nabla\varphi|^2 < \infty.
  \ee
  Then 
  \be{101.3}
	\bigg\{ \ir |\nabla\varphi|^2 \bigg\}^2 \le \ir \varphi |\Delta\varphi|^2.
  \ee
\end{lem}
\proof
  We only need to consider the case when the expression on the right-hand side of (\ref{101.3}) is finite, in which
  using the cut-off functions from (\ref{xi}) we may integrate by parts to find that
  \be{101.4}
	\ir \zeta _R^2 |\nabla\varphi|^2
	= - \ir \zeta _R^2 \varphi\Delta\varphi
	- 2\ir \zeta _R (\nabla\zeta _R\cdot\nabla\varphi)\varphi
	\qquad \mbox{for all } R>0,
  \ee
  where by the Cauchy-Schwarz inequality and (\ref{101.2}),
  \bea{101.5}
	- \ir \zeta _R^2 \varphi\Delta\varphi
	&\le& \bigg\{ \ir \zeta _R^2 \varphi|\Delta\varphi|^2 \bigg\}^\frac{1}{2} \cdot
	\bigg\{ \ir \zeta _R^2 \varphi \bigg\}^\frac{1}{2} \nn\\
	&\le& \bigg\{ \ir \varphi|\Delta\varphi|^2 \bigg\}^\frac{1}{2}
	\qquad \mbox{for all } R>0,
  \eea
  because $\zeta _R^2 \le 1$ on $\R^n$.
  By the same token and the observations that $\nabla\zeta _R\equiv 0$ in $B_R \cup (\R^n\setminus B_{R+1})$ and
  $|\nabla\zeta _R|\le 1$ in $B_{R+1}\setminus B_R$, again using he Cauchy-Schwarz inequality and (\ref{101.2})
  we see that moreover
  \bas
	\bigg| - 2\ir \zeta _R(\nabla\zeta _R\cdot\nabla\varphi)\varphi \bigg|
	&\le& 2\int_{B_{R+1}\setminus B_R} \varphi |\nabla\varphi| \\
	&\le& 2 \bigg\{ \int_{B_{R+1}\setminus B_R} \varphi|\nabla\varphi|^2 \bigg\}^\frac{1}{2} \cdot
	\bigg\{ \int_{B_{R+1}\setminus B_R} \varphi \bigg\}^\frac{1}{2} \\
	&\le& 2 \bigg\{ \int_{B_{R+1}\setminus B_R} \varphi|\nabla\varphi|^2 \bigg\}^\frac{1}{2}
	\qquad \mbox{for all } R>0,
  \eas
  so that our assumption that $\varphi|\nabla\varphi|^2$ belong to $L^1(\R^n)$ warrants that
  \be{101.6}
	- 2\ir \zeta _R(\nabla\zeta _R\cdot\nabla\varphi)\varphi \to 0
	\qquad \mbox{as } R\to\infty.
  \ee
  Since finally
  \bas
	\ir \zeta _R^2 |\nabla\varphi|^2
	\to \ir |\nabla\varphi|^2
	\qquad \mbox{as } R\to\infty
  \eas
  by e.g.~Beppo Levi's theorem, we thus obtain from (\ref{101.4})-(\ref{101.6}) that
  indeed (\ref{101.3}) holds.
\qed
\begin{lem}\label{lem102}
  Let $t_0\ge 0$ and $T>t_0$, and suppose that
  \be{102.1}
	w\in C^{2,1}(\R^n\times [t_0,T]) 
	\qquad \mbox{with}\quad [t_0,T] \ni t \mapsto \ir |\nabla w(\cdot,t)|^2 \in C^0([t_0,T])
  \ee
  is a nonnegative classical solution of
  \be{102.2}
	w_t = w\Delta w + w\ir |\nabla w|^2,
	\qquad x\in\R^n, t \in [t_0,T],
  \ee
  which is such that
  \be{102.3}
	\int_{t_0}^T \ir w|\Delta w|^2 < \infty
  \ee
  and 
  \be{102.4}
	\int_{t_0}^T \ir w|\nabla w|^2 < \infty
  \ee
  as well as
  \be{102.5}
	\ir w(\cdot,t) \le 1
	\qquad \mbox{for all } t\in (t_0,T).
  \ee
  Then 
  \be{102.6}
	\ir |\nabla w(\cdot,T)|^2 \le \ir |\nabla w(\cdot,t_0)|^2.
  \ee
\end{lem}
\proof
  For $R>0$, we again take $\zeta _R\in W^{1,\infty}(\R^n)$ as defined in (\ref{xi}) and integrate by parts using
  (\ref{102.2}) to compute
  \bea{102.7}
	& & \hspace*{-30mm}
	\frac{1}{2} \ir \zeta _R^2 |\nabla w(\cdot,T)|^2
	- \frac{1}{2} \ir \zeta _R^2 |\nabla w(\cdot,t_0)|^2 \nn\\
	&=& \int_{t_0}^T \ir \zeta _R^2 \nabla w \cdot \nabla w_t \nn\\
	&=& - \int_{t_0}^T \ir \nabla \cdot (\zeta _R^2 \nabla w)w_t \nn\\
	&=& - \int_{t_0}^T \ir \zeta _R^2 w|\Delta w|^2
	+ \int_{t_0}^T \Big\{ \ir \zeta _R^2 w\Delta w \Big\} \cdot \Big\{ \ir |\nabla w|^2 \Big\} \nn\\
	& & -2\int_{t_0}^T \ir \zeta _R (\nabla\zeta _R\cdot\nabla w) w\Delta w
	-2\int_{t_0}^T \Big\{ \ir \zeta _R (\nabla\zeta _R\cdot\nabla w) w\Big\} \cdot \Big\{ \ir |\nabla w|^2 \Big\}.
  \eea
  Here by the Cauchy-Schwarz inequality and (\ref{102.5}), for all $t\in (t_0,T)$ we estimate, again using that $\zeta _R^2 \le 1$, %in estimating
  \bas
	\ir \zeta _R^2 w\Delta w
	&\le& \bigg\{ \ir \zeta _R^2 w|\Delta w|^2 \bigg\}^\frac{1}{2} \cdot
	\bigg\{ \ir \zeta _R^2 w \bigg\}^\frac{1}{2} \\
	&\le& \bigg\{ \ir w|\Delta w|^2 \bigg\}^\frac{1}{2}
	\qquad \mbox{for all  } R>0,
  \eas
  so that since Lemma \ref{lem101} in conjunction with (\ref{102.4}) guarantees that
  \bas
	\ir |\nabla w|^2 \le \bigg\{ \ir w|\Delta w|^2 \bigg\}^\frac{1}{2}
	\qquad \mbox{for a.e.~} t\in (t_0,T),
  \eas
  it follows that
  \be{102.8}
	\int_{t_0}^T \Big\{ \ir \zeta _R^2 w\Delta w \Big\} \cdot \Big\{ \ir |\nabla w|^2 \Big\}
	\le \int_{t_0}^T \ir w|\Delta w|^2
	\qquad \mbox{for all } R>0.
  \ee
  Next, once more since $\supp\nabla\zeta _R \subset B_{R+1}\setminus B_R$ and $|\nabla\zeta _R| \le 1$ a.e.~in $\R^n$,
  by means of the Cauchy-Schwarz inequality we find that
  %\bea{102.9}
 \begin{align}\label{102.9}
	-2 \int_{t_0}^T \ir \zeta _R^2 (\nabla\zeta _R\cdot\nabla w) w\Delta w
	&\le 2 \int_{t_0}^T \int_{B_{R+1}\setminus B_R} w|\nabla w|\cdot|\Delta w| \nn\\
	&\le 2 \bigg\{ \int_{t_0}^T \int_{B_{R+1}\setminus B_R} w|\Delta w|^2 \bigg\}^\frac{1}{2} \cdot
	\bigg\{ \int_{t_0}^T \int_{B_{R+1}\setminus B_R} w|\nabla w|^2 \bigg\}^\frac{1}{2}
%	\qquad \mbox{for all } R>0
\end{align}
%  \eea
  and that, again thanks to (\ref{102.5}),
  \bea{102.10}
	-2 \int_{t_0}^T \Big\{ \ir \zeta _R (\nabla\zeta _R\cdot\nabla w) w\Big\} \cdot \Big\{ \ir |\nabla w|^2 \Big\}
	&\le& 2\int_{t_0}^T \Big\{ \int_{B_{R+1}\setminus B_R} w|\nabla w| \Big\} \cdot \Big\{ \ir |\nabla w|^2 \Big\} \nn\\
	&\le& 2c_1 \int_{t_0}^T \int_{B_{R+1}\setminus B_R} w|\nabla w| \nn\\
	&\le& 2c_1 \bigg\{ \int_{t_0}^T \int_{B_{R+1}\setminus B_R} w|\nabla w|^2 \bigg\}^\frac{1}{2} \cdot
	\bigg\{ \int_{t_0}^T \int_{B_{R+1}\setminus B_R} w \bigg\}^\frac{1}{2} \nn\\
	&\le& 2c_1 \sqrt{T} \bigg\{ \int_{t_0}^T \int_{B_{R+1}\setminus B_R} w|\nabla w|^2 \bigg\}^\frac{1}{2}
%	\qquad \mbox{for all } R>0,
  \eea
  for all $R>0$,
  where $c_1:=\|\nabla w\|_{L^\infty((t_0,T);L^2(\Omega))}$ is finite according to the second regularity requirement
  contained in (\ref{102.1}).
  In summary, from (\ref{102.7})-(\ref{102.10}) we infer that
  \bea{102.11}
	\frac{1}{2} \ir \zeta _R^2 |\nabla w(\cdot,T)|^2
	+ \int_{t_0}^T \ir \zeta _R^2 w|\Delta w|^2
	&\le& \frac{1}{2} \ir \zeta _R^2 |\nabla w(\cdot,t_0)|^2 \nn\\
	& & + \int_{t_0}^T \ir w|\Delta w|^2 \nn\\
	& & + 2 \bigg\{ \int_{t_0}^T \int_{B_{R+1}\setminus B_R} w|\Delta w|^2 \bigg\}^\frac{1}{2} \cdot
	\bigg\{ \int_{t_0}^T \int_{B_{R+1}\setminus B_R} w|\nabla w|^2 \bigg\}^\frac{1}{2} \nn\\
	& & +2c_1 \sqrt{T} \bigg\{ \int_{t_0}^T \int_{B_{R+1}\setminus B_R} w|\nabla w|^2 \bigg\}^\frac{1}{2} 
	\quad \mbox{for all } R>0.
  \eea
  Now since Beppo Levi's theorem ensures that as $R\to\infty$ we have
  \bas
	\ir \zeta _R^2 |\nabla w(\cdot,t)|^2 \to \ir |\nabla w(\cdot,t)|^2
	\qquad \mbox{for all } t\in [t_0,T]
  \eas
  and
  \bas
	\int_{t_0}^T \ir \zeta _R^2 w|\Delta w|^2
	\to \int_{t_0}^T \ir w|\Delta w|^2,
  \eas
  and since the inclusions $w|\Delta w|^2 \in L^1(\R^n\times (t_0,T))$ and 
  $w|\nabla w|^2 \in L^1(\R^n\times (t_0,T))$ asserted by (\ref{102.3}) and (\ref{102.4}) entail that
  \bas
	\int_{t_0}^T \int_{B_{R+1}\setminus B_R} w|\Delta w|^2 \to 0
	\quad \mbox{and} \quad
	\int_{t_0}^T \int_{B_{R+1}\setminus B_R} w|\nabla w|^2 \to 0
	\qquad \mbox{as } R\to\infty,
  \eas
  on taking $R\to\infty$ in (\ref{102.11}) we conclude that
  \bas
	\frac{1}{2} \ir |\nabla w(\cdot,T)|^2
	+ \int_{t_0}^T \ir w|\Delta w|^2
	\le
	\frac{1}{2} \ir |\nabla w(\cdot,t_0)|^2
	+ \int_{t_0}^T \ir w|\Delta w|^2,
  \eas
  and that hence (\ref{102.6}) is valid.
\qed
\begin{lem}\label{lem:w}
  Let $R>0$. The the solution $\varphi_R\in C^2(\Bbar_R)$ of $-\Delta \varphi_R=1$ in $B_R$ with 
  $\varphi_R=0$ on $\partial B_R$ is given by
  \bas
	\varphi_R(x):=\frac{R^2-|x|^2}{2n}, \qquad x\in \Bbar_R,
  \eas
  and hence for each $R>0$ we have
  \bas
  	\varphi_R(x)=R^2 \varphi_1\left(\frac xR\right) \quad \mbox{for all }
	x\in B_R \quad\mbox{and}\quad 
	\norm[L^1(B_R)]{\varphi_R}=\frac{2\omega_n}{n^2(n+1)} R^{n+2},
  \eas
  where $\omega_n:=|\partial B_1|$.
\end{lem}
\proofc of Proposition \ref{prop103}.\quad
  Since for all $t_0>0$ and $T>0$, Theorem \ref{theo10} asserts that $u\in C^{2,1}(\R^n \times [t_0,T])$ and that
  $[t_0,T] \ni t\mapsto \ir |\nabla u(\cdot,t)|^2$ is continuous, whereas Lemma \ref{lem9} says %\ref{lem5} and Lemma \ref{lem7}   say 
  that $u|\Delta u|^2$ and $u|\nabla u|^2$ belong to $L^1(\R^n\times (t_0,T))$, %\red{Wirklich? Wäre nicht der Verweis auf Lemma \ref{lem9} passender? Immerhin ist dort die in Rede stehende Lösung konstruiert worden (und es enthält direkt diese Aussage), während die genannten Lemmata nur $v$ betreffen.}
  the claimed monotonicity property is a consequence of Lemma \ref{lem102}.\\
  If (\ref{103.2}) was false, we could thus find $c_1>0$ such that $\ir |\nabla u(\cdot,t)|^2 \ge c_1$ for all
  $t>0$, which in view of (\ref{0}) would imply that
  \bas
	u_t \ge u\Delta u + c_1 u
	\qquad \mbox{for all $x\in\R^n$ and } t>0.
  \eas
  In order to argue from this to a contradiction to the fact that $(0,\infty)\ni t \mapsto \ir u(\cdot,t)$ is bounded,
  let us fix some conveniently large $R>0$ such that for the function $\varphi_R$ from Lemma \ref{lem:w} we have
%abbreviating
%  \be{103.33}
%	\varphi_R(x):=\frac{R^2-|x|^2}{2n}, \qquad x\in \bar B_R,
%  \ee
%  we have 
  \be{103.4}
	\frac{c_1}{2} \int_{B_R} \varphi_R >1.
  \ee
  Then since $u_0$ was assumed to be positive throughout $\R^n$, it is possible to find some suitably small 
  $c_2\in (0,\frac{c_1}{2})$ such that
  \be{103.5}
	u_0(x)>c_2\varphi_R(x)
	\qquad \mbox{for all } x\in \bar B_R,
  \ee
  whereupon we define
  \bas
	\uu(x,t):=y(t)\varphi_R(x),
	\qquad x\in\bar B_R, \ t\ge 0,
  \eas
  with
  \be{103.55}
	y(t):=c_2 e^{\frac{c_1}{2}t}, 
	\qquad t\ge 0.
  \ee
  Then (\ref{103.5}) warrants that $u(\cdot,0)>\uu(\cdot,0)$ in $\bar B_R$, whereas clearly $u>\uu=0$ on
  $\partial B_R\times [0,\infty)$. 
  Since writing $t_0:=\frac{2}{c_1} \ln \frac{c_1}{2c_2}>0$ we have 
  \be{103.6}
	y(t) \le \frac{c_1}{2} = y(t_0) \qquad \mbox{for all } t\in (0,t_0),
  \ee
  it particularly follows from (\ref{103.55}) that $y' =\frac{c_1}2 y \le c_1y-y^2$ on $(0,t_0)$ and that hence
  \bas
	\uu_t-\uu \Delta \uu - c_1 \uu
	&=& y' \varphi_R - y^2 \varphi_R \Delta \varphi_R - c_1 y \varphi_R \\
	&=& \Big\{ y' +y^2 - c_1 y \Big\} \cdot\varphi_R \\
	&\le& 0
	\qquad \mbox{in } B_R \times (0,t_0),
  \eas
  because $-\Delta\varphi_R=1$ in $B_R$.	% by (\ref{103.33}).
  Consequently, a comparison principle (\cite{wiegner}) ensures that $\uu\le u$ in $\bar B_R \times [0,t_0]$
  and that thus, thanks to the latter relation in (\ref{103.6}) and (\ref{103.4}),
  \bas
	\ir u(\cdot,t_0)
	\ge \int_{B_R} \uu(\cdot,t_0)
	= y(t_0) \int_{B_R} \varphi_R
	>1.
  \eas
  This is incompatible with the identity $\ir u(\cdot,t_0)=1$ obtained in Theorem \ref{theo10} and thereby 
  establishes (\ref{103.2}).
\qed
\mysection{Logarithmic growth of $\E$ for algebraically decaying data}\label{sect5}
In the case when $u_0$ satisfies algebraic decay conditions, thanks to our rather precise knowledge on
the transformation functions appearing in Definition \ref{def:fcts} and Lemma \ref{lem:h} we can obtain
the asymptotic properties of $\E$ claimed in Theorem \ref{theo200} in quite a straightforward manner
from the following known result on large time behavior in (\ref{0v}).
\begin{lem}\label{lem:v_estimate}
  Let $v_0\in C^0(\R^n)\cap L^\infty(\R^n)$ be positive, and let $v$ denote the solution of (\ref{0v})
  from Lemma \ref{lem0}.\\
  i) \ If there exist $\gamma>n$ and $C_0>0$ fulfilling  
  \bas
 	v_0(x)\geq C_0\cdot (1+|x|)^{-\gamma} \qquad \mbox{for all } x\in \R^n,
  \eas
  then 
  \bas
 	\norm[L^1(\R^n)]{v(\cdot,t)}\geq c t^{-(\gamma-n)/(\gamma+2)} \qquad \mbox{for all } t>1
  \eas  
  with some $c>0$.\abs
  ii) \ If $v_0\in L^p(\R^n)$ for some $p\in(0,1)$, then one can find
  $C>0$ such that 
  \bas
 	\norm[L^1(\R^n)]{v(\cdot,t)} \leq C t^{-(1-p)/(1+\frac{2p}n)} \qquad \mbox{for all }t>0.
  \eas
\end{lem}
\proof 
  For i), see \cite[Theorem 1.6 (i)]{fast_growth1}; for ii), we refer to \cite[Theorem 1.2]{fast_growth1}.
\qed
Indeed, this quite directly entails the claimed properties of $\E$ when $u_0$ decays as indicated in Theorem
\ref{theo200}.\abs
\proofc of Theorem \ref{theo200}. \quad
  i) \ Letting $v$ denote the solution of (\ref{0v}) from Lemma \ref{lem0} emanating from $v_0:=u_0$,
  from Lemma \ref{lem:v_estimate} we obtain $c_1>0$ such that 
  $\norm[L^1(\R^n)]{v(\cdot,s)}\geq c_1 s^{-\alpha}$ for all $s>1$, where $\alpha:=\frac{\gamma-n}{\gamma+2}$, 
  so that by Theorem \ref{theo10}, \eqref{def_u} and \eqref{eq:identity_hstrichg} we have
  \[
 	1=\norm[L^1(\R^n)]{u(\cdot,t)}=g(t)\norm[L^1(\R^n)]{v(\cdot,h(t))}
	\geq c_1 g(t) h^{-\alpha}(t)=c_1 h'(t) h^{-\alpha}(t)
  \]
  whenever $h(t)>1$, that is, whenever $t>H(1)=:t_0$. This shows that 
  \begin{equation}\label{eq:ghalpha}
 	g(t)=h'(t)\leq \frac1{c_1} h^\alpha(t) \qquad \mbox{ for all }t>t_0
  \end{equation}
  and that thus
  \[
 	1\geq c_1 h'(t) h^{-\alpha}(t)=\frac {c_1}{1-\alpha} (h^{1-\alpha})'(t) \qquad \mbox{ for all } t>t_0,
  \]
  meaning that
  \[
	h^{1-\alpha}(t)-h^{1-\alpha}(t_0)=\int_{t_0}^t (h^{1-\alpha})'(\tau)d\tau 
	\leq \frac{1-\alpha}{c_1} (t-t_0)\qquad \mbox{ for all } t>t_0,
  \]
  that is,
  \[
  h(t) \leq \left(\frac{1-\alpha}{c_1} t + h^{1-\alpha}(t_0)\right)^{\frac{1}{1-\alpha}}\qquad \mbox{ for all } t>t_0.
  \]
  Accordingly, by \eqref{eq:ghalpha} and the validity of $(a+b)^\alpha\leq 2^\alpha(a^\alpha+b^\alpha)$ 
  for all $a,b,\alpha\ge 0$,
  \begin{equation}\label{eq:gleq}
	g(t)\leq \frac{1}{c_1}\left(\frac{1-\alpha}{c_1} t +h^{1-\alpha}(t_0)\right)^{\frac{\alpha}{1-\alpha}} 
	\leq \frac{2^\alpha(1-\alpha)^{\frac{\alpha}{1-\alpha}}}{c_1^{\frac1{1-\alpha}}}t^{\frac{\alpha}{1-\alpha}}
	+\frac{2^\alpha}{c_1}h^\alpha(t_0)\qquad \mbox{ for all } t>t_0.
  \end{equation}
  As $g(t)=e^{\int_0^t L(\tau)d\tau}$ for all $t>0$ thanks to \eqref{eq:identity_gstrichgL} and
  the fact that $g(0)=1$, \eqref{eq:gleq} turns into the inequality
  \[
 	\E(t)=\int_0^t L(\tau)d\tau = \ln(g(t)) 
	\leq \ln \left(\frac{2^\alpha(1-\alpha)^{\frac{\alpha}{1-\alpha}}}{c_1^{\frac1{1-\alpha}}}t^{\frac{\alpha}{1-\alpha}}
	+\frac{2^\alpha}{c_1}h^\alpha(t_0) \right)
	%\left(\frac{2^\alpha h^\alpha(t_0)}{c_1^\alpha} + 2^\alpha(\frac{1-\alpha}{c_1})^{\frac{\alpha}{1-\alpha}} 
	\qquad \mbox{ for all }t>t_0.
  \]
  Now since for 
  $t>\max\{t_0,\frac{c_1}{1-\alpha}h^{1-\alpha}(t_0)\}$	%\max\{t_0,\frac{c_1^\alpha}{1-\alpha}h^{1-\alpha}(t_0)\}$
  we have that $\frac{2^\alpha(1-\alpha)^{\frac{\alpha}{1-\alpha}}}{c_1^{\frac1{1-\alpha}}}
  t^{\frac{\alpha}{1-\alpha}}>\frac{2^\alpha}{c_1}h^\alpha(t_0)$, in view of the observation that
  \[
 	\frac{\alpha}{1-\alpha}=\frac{1}{\frac1\alpha-1}=\frac1{\frac{\gamma+2}{\gamma-n}-1}
	=\frac{\gamma-n}{\gamma+2-(\gamma-n)}=\frac{\gamma-n}{n+2},
  \]
  we obtain \eqref{eq:E_upperbd} with $C:=\ln(2^{\alpha+1}(1-\alpha)^{\frac{\alpha}{1-\alpha}}c_1^{-\frac1{1-\alpha}})$
  and $T:=\max\{t_0,\frac{c_1}{1-\alpha}h^{1-\alpha}(t_0)\}$.\abs
  ii) \  
  Given $\eps>0$, we define $p:=\frac{n}{\gamma-\eps}$ and observe that then \eqref{eq:u0leq} shows that 
  $u_0\in L^p(\R^n)$. 
  Now proceeding as in part i), we let $v$ be the solution of (\ref{0v}) from Lemma \ref{lem0} with
  $v_0:=u_0$, and first invoke Lemma \ref{lem:v_estimate} to obtain $c_2>0$ such that 
  \bas
 	1=\norm[L^1(\R^n)]{u(\cdot,t)}=g(t)\norm[L^1(\R^n)]{v(\cdot,h(t))}\leq c_2 g(t)h^{-\alpha}(t)
	=c_2 h'(t)h^{-\alpha}(t)\qquad \mbox{for all }t>0,
  \eas
  this time with $\alpha=\frac{1-p}{1+\frac{2p}n}$, 
  so that $g(t)\geq \frac1{c_2} h^\alpha(t)$ and $(h^{1-\alpha})'(t)\geq\frac{1-\alpha}{c_2}$ for all $t>0$. 
  Consequently, 
  \[
 	h^{1-\alpha}(t)=\int_0^t(h^{1-\alpha})'(\tau) d\tau \geq \frac{1-\alpha}{c_2}t, \qquad \mbox{for all }t>0
  \]
  and 
  \[
 	g(t)\geq\frac1{c_2} h^\alpha(t)\geq c_2^{-\frac1{1-\alpha}} (1-\alpha)^{\frac{\alpha}{1-\alpha}} 
	t^{\frac{\alpha}{1-\alpha}} \qquad \mbox{for all }t>0.
  \]
  Since 
  \[
 	\frac{\alpha}{1-\alpha} = \frac{1}{\frac1\alpha-1} = \frac{1}{\frac{1+\frac{2p}n}{1-p}-1} 
	= \frac{1-p}{\frac{2p}n+p} =\frac{1-p}{p}\cdot\frac{n}{2+n},
  \]
  from the identity $\E(t)=\ln(g(t))$, $t>0$, we obtain that
  \bas
	\E(t)\geq \frac{1-p}p\cdot \frac{n}{2+n} \ln(t) - c_3 \qquad \mbox{for any }t>0
  \eas
  with $c_3:=\frac1{1-\alpha}\left(\ln(c_2)- \alpha\ln(1-\alpha)\right)$.
  Since herein we have
  \bas
 	\frac{(1-p)n}{p}=\frac{(1-\frac{n}{\gamma-\eps})n}{\frac{n}{\gamma-\eps}}
	=(\gamma-\eps)\Big(1-\frac{n}{\gamma-\eps}\Big)
	=\gamma-n-\eps
  \eas
  according to our definition of $p$, this yields the claimed inequality.
\qed
\section{Weakly sublinear growth of $\E$ for rapidly decreasing initial data}\label{sect6}
\subsection{A lower estimate for $\E$ in terms of $\calL$}
We next intend to examine how far imposing faster decay conditions on the initial data in (\ref{0}) can enforce
asymptotics of $\E$ different from that observed before for algebraically decreasing data.
Our first step in this direction yields a lower estimate for $\E$ under a given hypothesis on the temporal decay
of the solution to (\ref{0v}), formulated in terms of a rather general function $\ell$ fulfilling
appropriately mild conditions. 
\begin{lem}\label{lem:rapid_Egeq_preparation} 
  Assume that $\ell\in C^0([0,\infty))$ satisfies \eqref{ell:zero_pos_nondec}, \eqref{ell:monLstrich} and \eqref{ell:intcond}. 
  Let $u_0\in C^0(\R^n)$ be a positive function satisfying (\ref{10.1})-(\ref{10.3}), and
  let $u$ and $v$ be as in Lemma \ref{lem9}.
  Suppose that there exist $C_0>0$ and $s_0>0$ such that 
  \begin{equation}\label{eq:decreasecond_v}
 	\|v(\cdot,s)\|_{L^1(\R^n)} \le C_0 s^{-1} \cdot \ell^{-\frac{n+2}{n}}\Big(\frac{1}{s}\Big)
	\qquad \mbox{for all } s>s_0.
  \end{equation}
  Then we can find $t_0>0$ and $C_1 >0$, $C_2\ge 0$ such that with $\calL$ as defined in \eqref{eq:definecalL} we have  
  \begin{equation}\label{eq:Egeq_rapid}
  	\E(t)  \geq \ln\Big((\calL^{-1})'(C_1  t)\Big) -C_2
	\qquad \mbox{for all } t>t_0.
  \end{equation}
\end{lem}
\proof
  From \eqref{eq:decreasecond_v} we obtain that due to Theorem \ref{theo10} and \eqref{def_u},% and \eqref{eq:identity_hstrichg},
  \bas
  	1=g(t)\norm[L^1]{v(\cdot,h(t))}\leq C_0 g(t) \frac1{h(t)} \ell^{-\frac{n+2}n}\left(\frac1{h(t)}\right) 
	\qquad \mbox{for }t>t_1:=\max\set{H(s_0),H(1)},
  \eas
  which by \eqref{eq:definecalL} implies that
  \begin{equation}\label{eq:ggeq1dLstrich}
  	g(t)\geq \frac1{C_0}\cdot h(t) \ell^{\frac{n+2}n}\left(\frac1{h(t)}\right)=\frac{1}{C_0\calL'(h(t))} \qquad \mbox{for } t>t_1
  \end{equation}
  as well as, by \eqref{eq:identity_hstrichg}, 
  \bas
  	(\calL\circ h)'(t)\geq \frac{1}{C_0} \qquad \mbox{for } t>t_1.
  \eas
  Taking into account that our choice of $t_1$ ensures that $h(t_1)\geq 1$ and that thus $\calL(h(t_1))\geq\calL(1)=0$, 
  we have
  \bas
  	h(t)\geq \calL^{-1}\left(\calL(h(t_1))+\frac{t-t_1}{C_0} \right)
	\geq \calL^{-1}\left(\frac{t-t_1}{C_0} \right) 
	\qquad \mbox{for all }t>t_1.
  \eas
  According to \eqref{ell:monLstrich}, $\frac{1}{\calL'(\cdot)}$ is monotone on $(\xi_0,\infty)$ and hence 
  \eqref{eq:ggeq1dLstrich} implies that
  \bas
  	\E(t) = \ln(g(t)) 
	\ge \ln\left(\frac{1}{C_0\calL'\left(\calL^{-1}\left(c_1(t-t_1) \right)\right)}\right)
	=\ln\Big((\calL^{-1})'(c_1 t-c_2)\Big)-\ln C_0 
	\qquad \mbox{for all } t>t_2,
  \eas
  where $c_1:=\frac1{C_0}$, $c_2:=c_1t_1$ and $t_2:=\max\set{t_1,\frac1{c_1}(\calL(\xi_0)+c_2)}$.
  Once more employing monotonicity of $(\calL^{-1})'$,
  we obtain \eqref{eq:Egeq_rapid} with $C_1 :=\frac{c_1}2$, $C_2=\max\set{0,\ln C_0}$ and $t_0:=\max\set{\frac{2c_2}{c_1},2c_0\calL(\xi_0),t_2}$.
\qed
We proceed to ensure \eqref{eq:decreasecond_v} under appropriate conditions, relying on two
statements on decay of solutions to (\ref{0v}) in $L^q(\R^n)$ for small $q>0$ and in $L^\infty(\R^n)$, respectively,
as derived in \cite{fast_growth2}.
\begin{lem}\label{lem:rapiddecr_l1}
  Suppose that $\ell\in C^0([0,\infty))$ satisfies \eqref{ell:zero_pos_nondec} and 
  is such that with some $\xi_2>0, a>0$ and $\lambda_0>0$ we have $\ell \in C^2((0,\xi_2))$ and that \eqref{ell3} and \eqref{ell33} are satisfied. 
%   \be{l1_ell3}
% 	\xi\ell''(\xi) \ge - \ell'(\xi)
% 	\qquad \mbox{for all } \xi\in (0,\xi_2)
%   \ee
%   as well as
%   \be{l1_ell33}
% 	\ell(\xi)\le (1+a\lambda) \ell(\xi^{1+\lambda}) 
% 	\qquad \mbox{ for all } \xi\in(0,\xi_2) \mbox{ and each } \lambda\in(0,\lambda_0).
%   \ee
%   \red{Diese beiden Bedingungen stimmen mit \eqref{ell3} bzw. \eqref{ell33} überein. Sollen wir sie hier wiederholen oder wäre eine Auflistung 'satisfies \eqref{ell:zero_pos_nondec}, \eqref{ell3}, \eqref{ell33}. Moreover, ' ausreichend?}
  Moreover, assume that $v_0\in C^0(\R^n)$ is positive, radially symmetric and
  nonincreasing with respect to $|x|$. 
  Then if there exists $q_0>0$ such that
  \be{lq1}
 	v_0<\min\set{\xi_2^2,\xi_2^{\frac2{1+q_0}}} \qquad \mbox{in }\R^n,
  \ee
  and if furthermore
  \be{lq2}
 	\ir \ell(v_0) <\infty,
  \ee
  one can find $s_0>0$ and $C>0$ such that the minimal solution $v$ of \eqref{0v} satisfies 
  \bas
	\|v(\cdot,s)\|_{L^1(\R^n)} \le C s^{-1} \ell^{-\frac{n+2}{n}}\Big(\frac{1}{s}\Big)
	\qquad \mbox{for all } s> s_0.
  \eas
\end{lem}
\proof
  From \cite[Lemma 3.6]{fast_growth2} and Fatou's lemma 
  we infer that (\ref{lq1}) and (\ref{lq2}) ensure the existence of some $q\in (0,\min\{q_0,1\})$ 
  as well as positive constants $c_1$ and $s_1$ such that
  \be{l1.1}
 	\|v(\cdot,s)\|_{L^q(\R^n)} \le c_1 s^{-1} \ell^{-\frac{n+2q}{nq}}\Big(\frac{1}{s}\Big)
	\qquad \mbox{for all } s\geq s_1.
  \ee
  Moreover, the additional assumption on radial symmetry and monotonicity of $v_0$ enables us to invoke
  \cite[Theorem 1.3]{fast_growth2} which yields $c_2>0$ and $s_2>0$ fulfilling
  \be{l1.2}
 	\|v(\cdot,s)\|_{L^\infty(\R^n)} \le 
	c_2 s^{-1} \ell^{-\frac{2}{n}}\Big(\frac{1}{s}\Big)
	\qquad \mbox{for all } s\ge s_2.
  \ee
  Combining (\ref{l1.1}) with (\ref{l1.2}) we thus obtain
  \bas
 	\ir v(\cdot,s) 
	&=& \ir v^{1-q}(\cdot,s)v^q(\cdot,s) \\
	&\leq& \|v(\cdot,s)\|_{L^\infty(\R^n)}^{1-q} \|v(\cdot,s)\|_{L^q(\R^n)}^q \\
 	&\leq& c_1^{1-q} s^{-(1-q)} \ell^{-\frac{2}{n}(1-q)}\Big(\frac{1}{s}\Big)
	\cdot c_2^q s^{-q} \ell^{-\frac{n+2q}{nq}q} \Big(\frac{1}{s}\Big) \\
 	&=& c_1^{1-q} c_2^q s^{-1} \ell^{-\frac{n+2}n} \Big(\frac{1}{s}\Big)
  \eas
  for all $s>\max\{s_1,s_2\}$.
\qed
\subsection{An upper estimate for $\E$ in terms of $\calL$}
The following conditional information on $\E$ can be viewed as a counterpart of Lemma \ref{lem:rapid_Egeq_preparation}.
\begin{lem}\label{lem:rapid_Eleq_preparation}
  Let $\ell \in C^0([0,\infty))$ satisfy \eqref{ell:zero_pos_nondec}, \eqref{ell:monLstrich} and \eqref{ell:intcond}, 
  let $u_0\in C^0(\R^n)$ be a positive function satisfying (\ref{10.1})-(\ref{10.3}), and let $u$ and $v$
  be as introduced in Lemma \ref{lem9}. Then if
  there exist $C_0>0$ and $t_0>0$ such that 
  \begin{equation}\label{eq:cond_vgeq}
 	\norm[L^1(\R^n)]{v(\cdot,s)} \geq C_0 s^{-1} \ell^{-\frac{n+2}{n}} \Big(\frac{1}{s}\Big)
	\qquad \mbox{for all } s> s_0.
  \end{equation}
  we can find $t_1>0$ and $C_1 >0$, $C_2\ge 0$ with the property that for the function $\E$ from (\ref{def_E}) we have
  \begin{equation}\label{eq:Eleq_rapid}
  	\E(t)  \leq \ln\Big((\calL^{-1})'(C_1  t)\Big) +C_2
	\qquad \mbox{for any }t>t_0.
  \end{equation}
\end{lem}
\proof
  Proceeding similarly to the proof of Lemma \ref{lem:rapid_Egeq_preparation}, we use \eqref{eq:cond_vgeq}, 
  Theorem \ref{theo10}, \eqref{def_u}, \eqref{eq:definecalL} and \eqref{eq:identity_hstrichg} to  
  obtain %$c_1>0$ and $t_1>1$ satisfying 
  \bas
 	g(t)\leq \frac{1}{C_0\calL'(h(t))} 
	\quad \mbox{and}\quad 
	(\calL\circ h)'(t)\leq\frac1{C_0}
	\qquad \mbox{ for } t>t_1:=H(s_0), 
  \eas
  so that 
  \bas
 	\calL(h(t))\leq \calL(h(t_1))+\frac{t-t_1}{C_0}\leq c_1 t
	\qquad \mbox{for all } t>t_2:=\max\set{t_1,C_0\calL(h(t_1))}
  \eas
  with $c_1:=\frac{2}{C_0}$, whence finally
  \bas
 	\E(t)\leq\ln\left((\calL^{-1})'(c_1 t)\right), \qquad \mbox{whenever }t>\max\set{t_2,H(\xi _0)},%\frac1{c_2}\calL(\xi_0)},
  \eas
  as desired, with $\xi_0$ as in \eqref{ell:monLstrich}.
\qed
In order to identify conditions ensuring \eqref{eq:cond_vgeq}, we make use of a comparison argument inspired
by that in \cite[Theorem 1.6]{fast_growth2} to establish a pointwise lower estimate %from below
 for positive solutions of (\ref{0v}).
\begin{lem}\label{lem:rapid_lowerest}
  Let $\ell \in C^0([0,\infty)$ be such that \eqref{ell:zero_pos_nondec} holds, that $\ell$ is strictly increasing on $(0,\xi_1)$ 
  for some $\xi_1\in(0,\infty]$, and that
  \be{rle_ell:0limit}
 	\frac{\xi \ell'(\xi)}{\ell(\xi)} \to 0
	\qquad \mbox{as } \xi\searrow 0.
  \ee
  Moreover, assume that there exist $q\in (0,1)$ and  
  $\Rst>\max\set{\frac1{\sqrt[n]{\xi_1}},\frac1{\sqrt[n]{\lim_{\xi\nearrow \xi_1} \ell(\xi)}}}$ %schreibweise mit lim für den Fall \xi _1=\infty
  such that the positive function $v_0\in C^0(\R^n)\cap L^\infty(\R^n)$ satisfies 
  \be{le1}
	v_0(x)\geq \bigg\{ \ell^{-1}\left(\frac{1}{|x|^n}\right)\bigg\}^q
	\qquad \mbox{ for all }x\in\R^n\setminus B_{\Rst}.
  \ee
%  for $q\in(0,1)$ as in \eqref{ell:0limit}. 
  Then with some $s_0>0$ and $C>0$, the corresponding minimal solution 
  $v$ of \eqref{0v} fulfills
  \be{le2}
  	\|v(\cdot,s)\|_{L^1(\R^n)}
	\ge C s^{-1} \cdot \ell^{-\frac{n+2}n} \Big(\frac{1}{s}\Big)
	\qquad \mbox{for all } s>s_0.
  \ee
\end{lem}
\proof
  We first observe that 
  \bas
	z(x,\sigma):=(s+1) v(x,s),
	\qquad x\in\R^n, \ \sigma=\ln (s+1)\ge 0,
  \eas
  defines a positive classical solution of
  \bas
	\left\{ \begin{array}{l}
	z_\sigma=z\Delta z + z,
	\qquad x\in\R^n, \ \sigma>0, \\[1mm]
	z(x,0)=v_0(x),
	\qquad x\in\R^n,
	\end{array} \right.
  \eas
  and in order to estimate $z$ from below appropriately, given $\sigma>0$ we set
  \begin{equation}\label{def:Rsigma}
	R(\sigma):=\Big\{ \ell(e^{-\sigma})\Big\}^{-\frac{1}{n}},
  \end{equation}
  and then obtain from L'Hospital's rule that \eqref{rle_ell:0limit} entails that
  \newcommand{\limtau}{\lim_{\sigma\to\infty}}
  \bas
  	\limtau \frac{2}{\sigma} \ln R(\sigma) = \limtau \frac{-2}{n\sigma} \ln(\ell(e^{-\sigma})) 
  	=\limtau \frac2n \frac{\ell'(e^{-\sigma})e^{-\sigma}}{\ell(e^{-\sigma})}
  	=\lim_{\xi \searrow0} \frac{\xi \ell'(\xi )}{\ell(\xi )}=0.
  \eas
  Hence,
  \bas
  	\limtau \Big((q-1)\sigma + 2\ln R(\sigma)\Big)=
	\limtau\Big(-(1-q)+\frac2\sigma \ln R(\sigma)\Big)\sigma=-\infty,
  \eas
  so that we can find $c_1>0$ with the property that
  \be{le5}
 	(q-1)\sigma+2\ln R(\sigma)\le c_1
	\qquad \mbox{for all } \sigma>0.
  \ee
  With $R_\star$ taken from (\ref{le1}), we now abbreviate
  \be{le3}
	\sigma_\star:=\frac{1}{q} \ln \frac{1}{\inf_{x\in B_{R_\star}} v_0(x)},
  \ee
  and for $\sigma_0>\sigma_\star$ we let
  \begin{equation}\label{def:deltasigma0}
	\delta(\sigma_0):=\frac{2n e^{-q\sigma_0}}{R^2(\sigma_0)}
  \end{equation}
  and
  \bas
	y(\sigma):=\Big( \frac{e^{-\sigma}}{\delta(\sigma_0)} + 1 - e^{-\sigma}\Big)^{-1},
	\qquad \sigma\ge 0,
  \eas
  that is, we let $y$ be the solution of
  \bas
	\left\{ \begin{array}{l}
	y'(\sigma)=y(\sigma)-y^2(\sigma),
	\qquad \sigma>0, \\[1mm]
	y(0)=\delta(\sigma_0).
	\end{array} \right.
  \eas
  Thus, if we define
  \bas
	\uz(x,\sigma):=y(\sigma) \varphi_{R(\sigma_0)}(x),
	\qquad x\in B_{R(\sigma_0)}, \ \sigma\ge 0,
  \eas
  with $\varphi_{R(\sigma_0)}$ as in Lemma \ref{lem:w}, then $\uz$ solves
  \bas
	\uz_\sigma- \uz\Delta \uz - \uz
	&=& y' \varphi_{R(\sigma_0)} - y^2 \varphi_{R(\sigma_0)} \Delta \varphi_{R(\sigma_0)} - y\varphi_{R(\sigma_0)} \\
	&=& (y-y^2) \varphi_{R(\sigma_0)} + y^2 \varphi_{R(\sigma_0)} - y\varphi_{R(\sigma_0)} \\[2mm]
	&=& 0,
	\qquad x\in B_{R(\sigma_0)}, \ \sigma>0,
  \eas
  and evidently $\uz(x,\sigma)=0\le z(x,\sigma)$ for all $x\in \partial B_{R(\sigma_0)}$ and $\sigma>0$.
  Moreover, using the montonicity of $\ell^{-1}$ in $[0,\frac1{\Rst^n})$ we may employ (\ref{le1}) to see that for all
  $x\in B_{R(\sigma_0)} \setminus B_{R_\star}$ %Die Menge ist nur dann nicht leer, wenn R(\sigma_0)>\Rst, weshalb wir das oBdA annehmen dürfen
  we can estimate
  \bas
	u_0(x)
	\ge \bigg\{ \ell^{-1} \Big(\frac{1}{R^n(\sigma_0)}\Big) \bigg\}^q
	= \bigg\{ \ell^{-1} \Big( \ell (e^{-\sigma_0})\Big) \bigg\}^q
	= e^{-q\sigma_0},
  \eas
  so that since
  \be{le4}
	\varphi_{R(\sigma_0)}(x) \le \frac{R^2(\sigma_0)}{2n}
	\qquad \mbox{in } B_{R(\sigma_0)},
  \ee
  according to \eqref{def:deltasigma0} for any such $x$ we have
  \bas
	\frac{z(x,0)}{\uz(x,0)}
	= \frac{v_0(x)}{\delta(\sigma_0) \varphi_{R(\sigma_0)}(x)}
	\ge \frac{e^{-q\sigma_0}}{\delta(\sigma_0) \cdot \frac{R^2(\sigma_0)}{2n}}
	= 1.
  \eas
  In the corresponding inner region, again by (\ref{le4}) we infer that 
  \bas
 	\frac{z(x,0)}{\uz(x,0)}
	&\ge& \frac{\inf_{x\in B_{R_\star}} v_0(x)}{\delta(\sigma_0)\cdot \frac{R^2(\sigma_0)}{2n}}
	= e^{q\sigma_0} \cdot \inf_{x\in B_{R_\star}} v_0(x)
	\ge e^{q\sigma_\star} \cdot \inf_{x\in B_{R_\star}} v_0(x)
	=1
	\qquad \mbox{for all } x\in B_{R_\star}
  \eas
  due to (\ref{le3}) and our restriction $\sigma_0>\sigma_\star$.
  A comparison argument (\cite{wiegner}) therefore shows that $z\geq \uz$ in $B_{R(\sigma_0)} \times (0,\infty)$ 
  and that thus, in particular,
  \bas
	\|z(\cdot,\sigma_0)\|_{L^1(\R^n)}
	&\ge& \|\uz(\cdot,\sigma_0)\|_{L^1(B_{R(\sigma_0)})} \\
	&=& y(\sigma_0) \|\varphi_{R(\sigma_0)}\|_{L^1(B_{R(\sigma_0)})} \\
	&=& \Big(\frac{e^{-\sigma_0}}{\delta(\sigma_0)} + 1 - e^{-\sigma_0}\Big)^{-1} \cdot
	\frac{2\omega_n}{n^2(n+2)} R^{2+n}(\sigma_0)
  \eas
  according to Lemma \ref{lem:w}.
  As our choice of $\delta(\sigma_0)$ together with (\ref{le5}) ensures that herein
  \bas
	\frac{e^{-\sigma_0}}{\delta(\sigma_0)} + 1 - e^{-\sigma_0}
	&\le& \frac{e^{-\sigma_0}}{\delta(\sigma_0)} + 1 \\
	&=& \frac{1}{2n} e^{-(1-q)\sigma_0 + 2\ln R(\sigma_0)} +1 \\
	&\le& \frac{1}{2n} e^{c_1}+1
	\qquad \mbox{for all } \sigma_0>\sigma_\star,
  \eas
  writing $c_2:=(\frac1ne^{2c_1}+1)^{-1}\frac{2\omega_n}{n^2(n+2)}$ we obtain 
  \bas
	\|z(\cdot,\sigma_0)\|_{L^1(\R^n)}
	\ge c_2 R^{2+n}(\sigma_0)
	\qquad \mbox{for all } \sigma_0>\sigma_\star.
  \eas
  For arbitrary $s>e^{\sigma_\star}-1$, choosing $\sigma_0:=\ln (s+1)$ we see that thanks to the monotonicity of
  $\ell$ and \eqref{def:Rsigma} this implies that
  \bas
	\|v(\cdot,s)\|_{L^1(\R^n)}
	&\ge& (s+1)^{-1}\|z(\cdot,\ln(s+1))\|_{L^1(\R^n)} \\
	&\ge& c_2(s+1)^{-1} \cdot\ell^{-\frac{n+2}{n}} \Big( e^{-\ln (s+1)}\Big) \\
	&\ge& c_2(s+1)^{-1} \cdot \ell^{-\frac{n+2}{n}}\Big(\frac{1}{s+1}\Big) \\
	&\ge& c_2(s+1)^{-1} \cdot \ell^{-\frac{n+2}{n}}\Big(\frac{1}{s}\Big) 
  \eas
  and thereby readily yields (\ref{le2}).
\qed
\subsection{Proof of Theorem \ref{theo300}}
We now only need to combine the results of the previous two sections to obtain our main result on asymptotic
behavior of $\E$ for rapidly decreasing initial data.\abs
\proofc of Theorem \ref{theo300}.\quad
  Part i) can be obtained by a straightforward combination of Lemma \ref{lem:rapid_Eleq_preparation}  with
  Lemma \ref{lem:rapid_lowerest} and a comparison argument (\cite{wiegner}), whereas part ii) similarly results from
  Lemma \ref{lem:rapid_Egeq_preparation} and Lemma \ref{lem:rapiddecr_l1}.
\qed
\subsection{Examples}
\subsubsection{Exponentially decaying data. Proof of Corollary \ref{cor:firstexample}}
%\alpha\in(0,1) brauchen wir zum Anwenden von Thm \ref{rapid}, c_0>2^\alpha dann, um diese Funktion \ell zu nutzen. Vielleicht ließe die sich anpassen (und ein M<1 sich nutzen), aber selbst dann gäbe es noch eine Bedingung an $c_0$: $c_0^\alpha=:M$, s.u.
%
%
%
%
%
%
%
\begin{proof}[\proofc of Corollary \ref{cor:firstexample}]
Given $\kappa>0$ and $M\geq 2$ we let 
\[
 \ell(\xi)=\begin{cases}0,&\xi=0,\\\ln^{-\kappa}\left(\frac M\xi\right),& \xi\in(0,\frac M2),\\\ln^{-\kappa}(2),&\xi\geq \frac M2.\end{cases}
\]
This function obviously satisfies \eqref{ell:zero_pos_nondec} and is continuous as well as strictly increasing on $(0,\xi_1)$ for $\xi_1=\frac M2$. That it moreover fulfills \eqref{ell3} and \eqref{ell33} has been shown in \cite[Lemma 3.9]{fast_growth2} for $\xi_2=\frac M2$, arbitrary $\lambda_0>0$ and $a=\kappa$ if $\kappa<1$ and $a=\frac{(1+\lambda_0)^{\kappa}-1}{\lambda_0}$ otherwise. This also entails \eqref{ell:0limit}, see \cite[Lemma 2.1]{fast_growth2}. %Remark \ref{rem:ell:0limit}).
% Moreover, \eqref{ell:0limit} holds, because due to $\ell'(s)=-\kappa \ln^{-\kappa-1}(\frac Ms)\frac{s}M M(-\frac1{s^2}=\frac1{s\ln^{\kappa+1}(\frac Ms)}$
% \[
%  \frac{s\ell'(s)}{\ell(s)} = \frac{\frac{s}s\ln^\kappa(\frac Ms)}{\ln^{\kappa+1}(\frac Ms)}=\frac1{\ln \frac Ms}\to 0 \qquad \mbox{as }s\searrow 0
% \]
%and 
Finally, there exists $\xi_0>0$ such that $\xi_0\leq \xi\mapsto \ell^{\frac{n+2}n}(\frac1\xi)=\xi \ln^{-\kappa \frac{n+2}n}(M\xi)$, i.e. \eqref{ell:monLstrich}.
%(Because $\frac{s}{\ln^\beta Ms}$ ($\beta=\kappa\frac{n+2}n$) has derivative $\frac{\ln^\beta Ms-\frac\beta M\ln^{\beta-1}(Ms)}{\ln^{2\beta} Ms}$, (where btw $\ln Ms>0$ for ``large'' $s$), which is positive, whenever $\ln Ms>\frac\beta M$) -- $\xi_0:=\frac1M exp(\frac\beta M)$

Furthermore, 
\[
 \int_1^\infty \frac1{\xi } \ell^{-\frac{n+2}n}\left(\frac1{\xi }\right) d\xi  =\int_1^\infty\frac1{\xi }\ln^{\kappa\frac{n+2}n}(M\xi )d\xi  =\infty%=\int_0^{\frac2M} \frac1s \ln^{\kappa\frac{n+2}n} 2 \leq\int_0^1 \frac1s \ell^{-\frac{n+2}n}(\frac1s) ds
\]
and thus \eqref{ell:intcond} holds.

The function $\ell^{-1}\colon [0,\ln^{-\kappa}2)\to [0,\frac M2)$ is given by $\ell^{-1}(\xi )=Me^{-\xi ^{-\frac1\kappa}}$ for $\xi\in(0,\ln^{-\kappa}2)$. 
We compute $\calL(t)=\int_1^t \frac1{\xi }\ell^{-\frac{n+2}n}(\frac1{\xi }) d\xi $ for $t>1$, so that we have 
\[
 \calL(t)=\int_1^t \frac1{\xi }\ln^{\kappa\frac{n+2}n}\left(M{\xi }\right) d\xi  = \int_{\ln M}^{\ln Mt} y^{\kappa\frac{n+2}n} dy = \frac1{1+\kappa\frac{n+2}n} \left[\ln^{1+\kappa\frac{n+2}n} Mt -\ln ^{1+\kappa\frac{n+2}n} M\right]
\]
%For $t\leq \frac2M$, on the other hand, we obtain 
% \[
%  \calL(t)=-\int_t^{\frac2M}\frac1{\xi } \ln^{\kappa\frac{n+2}n}(2) d\xi  - \int_{\frac2M}^1\frac1{\xi }\ln^{\kappa\frac{n+2}n}(M\xi )d\xi =-\ln^{\kappa\frac{n+2}n}2\ln\frac{2}{Mt}+\frac{1}{1+\kappa\frac{n+2}n}[\ln^{1+\kappa\frac{n+2}n}2-\ln^{1+\kappa\frac{n+2}n}M].
% \]
if $t>1\geq\frac2M$. In short, 
\[
 \calL(t)=c_1\ln^\gamma Mt-c_2,\qquad\text{for all } t>1%\geq \frac2M,%\\c_3+c_4\ln t,&t\leq\frac2M,\end{cases}
\]
where 
\begin{equation}\label{def:gamma}
\gamma:=1+\kappa\frac{n+2}n,\end{equation}
and $c_1:=\frac1\gamma$, and $c_2:=\frac1\gamma \ln^\gamma M$. %, $c_3=\frac1\gamma[\ln^\gamma2-\ln^\gamma M]-(\ln^{\kappa\frac{n+2}n}2)\cdot\ln\frac2M$, $c_4=\ln^{\kappa\frac{n+2}n}2$.
Accordingly, we have 
\[
 \calL^{-1}(t)=\frac1M \exp\left(\left(\frac{t+c_2}{c_1}\right)^\frac1\gamma\right)\qquad\text{for all } t\geq c_1\ln^\gamma M-c_2,%\\ exp(\frac{t-c_3}{c_4}),& t\leq c_3+c_4\ln\frac2M\end{cases}
\]
% c_1\ln^\gamma2-c_2=c_3+c_4\ln 2, 
% \iff
% \frac1\gamma \ln^\gamma 2-\frac1\gamma \ln^\gamma M = \frac1\gamma(\ln^\gamma 2-\ln^\gamma M)-\ln^{\kappa\frac{n+2}n} 2 \cdot \ln \frac2M + \ln^{\kappa\frac{n+2}n}2\cdot \ln\frac2M 
%\iff true
and
\[
 (\calL^{-1})'(t)=\frac1M \exp\left(\frac{t+c_2}{c_1}\right)^{\frac1\gamma} \frac1\gamma \left(\frac{t+c_2}{c_1}\right)^{\frac1\gamma-1} \frac1{c_1}\qquad \text{for all } t>c_1\ln^{\gamma} M-c_2 
%\\  \frac1{c_4} \exp(\frac{t-c_3}{c_4}),&t<c_3+c_4\ln\frac2M,\end{cases}
\]
that is 
\[
 \ln\left((\calL^{-1})'(t)\right)=\left(\frac{t+c_2}{c_1}\right)^{\frac1\gamma} + \left(\frac1\gamma -1\right)\ln \left(\frac{t+c_2}{c_1}\right) -\ln M %note: \gamma c_1=1 
\]
if $t>c_1\ln^{\gamma} M-c_2$, and thus there are positive constants $c_5$, $c_6$, and $t_0>0$ such that 
\begin{equation}\label{eq:estlncalLinv}
 c_5 t^{\frac1\gamma}\leq \ln (\calL^{-1})'(t)\leq c_6 t^{\frac1\gamma} \qquad \mbox{ for all } t>t_0.
\end{equation}
Having derived properties of $\ell$ for general parameters, we can now turn our attention to the proof of part \ref{cor:firstexample:leq} of Corollary \ref{cor:firstexample}. Namely, fixing $\Rst>0$, $\beta>0$, $c_0>0$ and $\alpha\in(0,1)$ such that $u_0(x)\geq c_0e^{-\alpha|x|^\beta}$ for all $x\in \R^n\setminus B_{\Rst}$, we take any $q\in(\alpha,1)$ and $M\geq 2$ and let $\kappa:=\frac{n}{\beta}>0$. 
By possibly enlarging $\Rst$ we ensure that $c_0M^{-q}\geq e^{-(q-\alpha)|x|^\beta}$ for all $x\in\R^n\setminus B_{\Rst}$, so that 
\[
 u_0(x)\geq c_0e^{-\alpha|x|^{-\beta}}\geq M^q e^{-q|x|^\beta}=\left[\ell^{-1}\left(\frac1{|x|^n}\right)\right]^q \quad \text{for all } x\in \R^n\setminus B_{\Rst}.
\]
In light of the attributes of $\ell$ previously asserted, from Theorem \ref{theo300} and \eqref{eq:estlncalLinv} we  obtain $t_0>0$, $C_1>0$, and $C_2\ge 0$ such that 
$ \E(t)\leq C_1t^{\frac1\gamma}+C_2$ for $t>t_0$, which in turn yields
\[
 \E(t)\leq C t^{\frac{1}{1+\frac{n+2}\beta}} \qquad \text{for } t>t_0, 
\]
if we set $C:=C_1+C_2t_0^{-\frac1\gamma}$ and take into account \eqref{def:gamma}. 

In order to prove Corollary \ref{cor:firstexample} \ref{cor:firstexample:geq}, we fix $\alpha>0$, $\beta>0$ and $C_0>0$ such that $u_0(x)\leq C_0e^{-\alpha|x|^\beta}$ for all $x\in \R^n$, and pick $\eps>0$. We choose $M>\max\set{2C_0,2}$ and let $\kappa:=\frac{n}{\beta}+\frac{n\eps}{n+2}$. These choices entail that, since $\xi_2=\frac M2>1$, for any $q_0\in (0,1)$ we have  
\[
 u_0(x)\leq \uo_0(x):=C_0e^{-\alpha|x|^\beta}\leq C_0\leq \frac M2=\xi_2<\xi_2^{\frac2{1+q_0}}<\xi_2^2 \qquad \text{for all } x\in\R^n, 
\]
and that $\kappa\beta>n$, so that 
\[
  \irn \ell(\uo(x))dx= \irn \ell\left(C_0e^{-\alpha |x|^\beta}\right)dx =\irn \ln^{-\kappa}\left(\frac{M}{C_0} e^{\alpha |x|^\beta}\right)dx =\irn \frac{1}{\left(\ln\frac M{C_0}+\alpha |x|^\beta\right)^{\kappa}}dx<\infty.
\]
Furthermore, according to \eqref{def:gamma}, $\gamma=1+\frac{n+2}{\beta}+\eps$. Consequently, Theorem \ref{theo300} becomes applicable and due to \eqref{eq:estlncalLinv} directly results in Corollary \ref{cor:firstexample} \ref{cor:firstexample:geq}.
% We have that $c_0e^{-\alpha|x|^\beta}\geq[\ell^{-1}(\frac{1}{|x|^n})]^q=M^q e^{-q|x|^{\frac n\kappa}}$ for suitable \green{$q$, $M$, and $\kappa$} if $\alpha=:q\in(0,1)$, $\frac{n}{\beta}=:\kappa>0$ and $c_0^{\frac1\alpha}=:M\geq2$. In this case, Theorem \ref{rapid} \ref{r_Eleq} becomes applicable, producing Corollary \ref{cor:firstexample} \ref{cor:firstexample:leq}.
% 
% If, given $c_0,\alpha,\beta>0$, we choose $M>2c_0$, then $c_0e^{-\alpha|x|^\beta}\leq c_0<\frac M2<\frac{M^2}4=\xi _2^2$ (and $c_0e^{-\alpha|x|^\beta}<\xi _2^{\frac2{1+q_0}}$ for some small $q_0>0$). Moreover, 
% \[
%  \irn \ell\left(c_0e^{-\alpha |x|^\beta}\right)=\irn \ln^{-\kappa}\left(\frac{M}{c_0} e^{\alpha |x|^\beta}\right) =\irn \frac{1}{\left(\ln\frac M{c_0}+\alpha |x|^\beta\right)^{\kappa}},
% \]
% which is finite whenever $\kappa\beta>n$, that is, when $\gamma=1+\kappa\frac{n+2}n>1+\frac{n+2}\beta$. 
% 
% Application of Theorem \ref{rapid} \ref{r_Egeq} therefore yields Corollary \ref{cor:firstexample} \ref{cor:firstexample:geq}.
\end{proof}
\subsubsection{Doubly exponentially decaying data. Proof of Corollary \ref{cor:secondexample}}
%
%
%
%
%
%
%
%
%We continue with a second family of functions, which looks similar and sheds light on the behavior of solutions emanating from initial data which decay like $e^{-e^{|x|^\beta}}$.\\
\begin{proof}[\proofc of Corollary \ref{cor:secondexample}]
Fixing $\beta>0$, $\alpha>0$, $C_0>0$ and $\eps>0$ as in the assumptions of Corollary \ref{cor:secondexample}, we first pick $\betatilde<\beta$ such that $\frac{n+2}{\betatilde}<\frac{n+2}{\beta}+\eps$ and choose $\alphatilde>1$ and $\Ctilde_0>C_0$ such that $C_0e^{-\alpha|x|^\beta}\leq \Ctilde_0 e^{-\alphatilde|x|^{\betatilde}}$ for all $x\in\R^n$. %e.g. let $\Ctilde_0:=c_0\exp(\sup_{z>0} \alphatilde z^\betatilde-\alpha z^\beta)$, 
Moreover, we define $\gamma:= \frac{n+2}{\beta}+\eps$, let %Fixing $\beta>0$, $\alpha>1$, $C_0>0$ and $\gamma>\frac{n+2}\beta$ as in the assumptions of Corollary \ref{cor:secondexample}, we let 
$M>\max\set{\Ctilde_0ee^{-\alphatilde},e,\Ctilde_0}$ and choose $\xi_2\in\left[\max\set{\Ctilde_0e^{-\alphatilde},1},\frac Me\right)$. Then 
\[
 u_0(x)\leq \uo_0(x):=\Ctilde_0e^{-\alphatilde e^{|x|^{\betatilde}}}\leq \Ctilde_0e^{-\alphatilde}<\xi_2<\xi_2^{\frac{2}{1+q_0}}<\xi_2^2 \qquad \text{for all } x\in \R^n
\]
for arbitrary $q_0\in(0,1)$. Moreover, we let $\kappa:=\frac{n\gamma}{n+2}$ and define 
%Given $\kappa>0$ and $M>e$, $\xi _2\in[1,\frac Me)$, we let 
\[
 \ell(\xi)=\begin{cases}0,&\xi=0,\\\ln^{-\kappa}\ln\left(\frac M\xi\right),& \xi\in(0,\xi_2),\\\ln^{-\kappa}\ln\left(\frac M {\xi_2}\right),&\xi\geq \xi_2.\end{cases}
\]
Then 
\[
 \irn \ell(\uo_0(x))dx=\irn \ln^{-\kappa}\ln \left(\frac{M}{\Ctilde_0} e^{\alphatilde e^{|x|^{\betatilde}}}\right)dx\leq \irn \frac{1}{\left(\ln \alphatilde + |x|^{\betatilde}\right)^{\kappa}}dx 
\]
is finite, because $\kappa\betatilde=\frac{\betatilde \gamma}{n+2}\cdot n > n$. 
Furthermore, \eqref{ell:zero_pos_nondec} and $\ell\in C^0([0,\infty))$ are apparently satisfied and $\ell$ is strictly increasing on $(0,\xi_2)$. In addition, \eqref{ell3} and \eqref{ell33} have been shown in \cite[Lemma 3.11]{fast_growth2}. By 
%Remark \ref{rem:ell:0limit}
\cite[Lemma 2.1]{fast_growth2}, this also implies \eqref{ell:0limit}.
% On $(0,s_2)$ we have $\ell'(s)=\frac{\kappa}{s\ln\frac Ms} \ln^{-\kappa-1}\frac Ms$ and hence \eqref{ell:0limit} is fulfilled according to 
% \[
%  \lim_{s\searrow0} \frac{s\ell'(s)}{\ell(s)}=\lim_{s\searrow 0} \frac{\ln^\kappa\ln\frac Ms}{\ln^{\kappa+1}\ln\frac Ms\ln\frac Ms}=\lim_{s\searrow0}\frac1{\ln\ln\frac Ms \ln \frac Ms}=\lim_{\sigma\nearrow\infty}\frac1{\sigma\ln\sigma}=0.
% \]
Moreover, \eqref{ell:monLstrich} holds, because $\xi\mapsto \xi\ln^{-\kappa\frac{n+2}n}\ln M\xi$ is nondecreasing for large values of $\xi$ (as can easily be seen by checking the sign of the derivative $\ln^{-\kappa\frac{n+2}n-1}\ln M\xi \cdot(\ln\ln M\xi-\frac{\kappa(n+2)}{n\ln M\xi})$). Also \eqref{ell:intcond} holds true, because 
% \[
%  \int_0^1\frac1s\ell^{-\frac{n+2}n}(\frac1s) ds \geq \int_0^{\frac1{s_2}} \frac1s \ln^{\kappa\frac{n+2}n} \ln \frac M{s_2}=\infty
% \]
% and
\[
 \int_1^\infty\frac1\xi\ell^{-\frac{n+2}n}\left(\frac1\xi\right) d\xi =\int_1^\infty \frac1\xi\ln^{\kappa\frac{n+2}n} \ln M\xi d\xi =\int_{\ln M}^\infty \ln^{\kappa\frac{n+2}n} y dy=\infty.
\]
The inverse $\ell^{-1}\colon [0,\ln^{-\kappa}\ln \frac M{\xi_2})\to [0,\xi_2)$ is given by $\ell^{-1}(\xi)=Me^{-e^{\xi^{-\frac1\kappa}}}$ for $\xi\in\left(0,\ln^{-\kappa}\ln \frac M{\xi_2}\right)$.\\
In this case it is more difficult to give an explicit expression for $\calL^{-1}$ than before. We provide a workaround in the following: 
Since $\gamma=\kappa\frac{n+2}n$, 
% \red{Nachdem die Aussage wieder zur H\"alfte \"uberfl\"ussig geworden ist, lohnt es sich kaum noch, sie als eigene Nummer hervorzuheben. Sollen wir auf gesonderte S\"atze, Lemmata u.\"a. in diesem Abschnitt, der ja insgesamt ein Beweis ist, verzichten?\\
% Oder sollten wir gar mehr in Lemmata verpacken?}
% \begin{lemmation}
%  Letting $f(t)=\ln\ln\calL^{-1}(t)$ and $\gamma=\kappa\frac{n+2}n$, we have that there are $t_0>0$, $c_1>0$ such that 
% \[
%  c_1 t\leq (f(t))^\gamma e^{f(t)} \qquad \mbox{for all }t>t_0.
% \]
% \end{lemmation}
% %
% %
% \proof
%  T
the function $\calL$ is given by 
\[
 \calL(t)=\int_{\ln M}^{\ln M+\ln t} \ln^\gamma y dy \qquad \text{for all } t>1. 
\]
% where the integrand $\ln^\gamma$ is concave on $(y_0,\infty)$ with $y_0=e^{\gamma-1}$. We use this concavity to estimate $\calL$:
% \begin{align*}
%  \calL(t)\geq&\int_{y_0}^{\ln M+\ln t} \ln^\gamma y dy +\int_{\ln M}^{y_0} \ln^\gamma ydy \geq(\ln t+\ln M-y_0) \frac12 (\ln^\gamma y_0 + \ln^\gamma(\ln M+\ln t))+c_0\\
%  \geq&\frac14\ln t \ln^\gamma(\ln M+\ln t)+c_0 \geq \frac15 \ln t \ln^\gamma(\ln t),
% \end{align*}
% as long as $\frac13\ln t>y_0-\ln M$ and $\frac14\ln^\gamma(\ln M+\ln t)>-\ln^\gamma y_0$ and $\frac15\ln t\ln^\gamma\ln t\geq -c_0$, and where we have set $c_0=\int_{\ln M}^{y_0} \ln^\gamma ydy$. 
We employ positivity and monotonicity of $\ln^\gamma$, and thereby may infer
\[
 \calL(t)\leq \int_1^{2\ln t} \ln^\gamma ydy\leq 2\ln t\ln^\gamma(2\ln t)\leq 2\ln t (\ln 2+\ln\ln t)^\gamma \leq 2\ln t (2\ln\ln t)^\gamma \leq 2^{1+\gamma}\ln t\ln^\gamma\ln t, 
\]
as long as $\ln t>\ln M$ and $\ln\ln t\geq \ln 2$. Inserting $\calL^{-1}(t)$ istead of $t$ and using $\calL(\calL^{-1}(t))=t$, we obtain that 
\[
 %\frac15\ln\calL^{-1}(t)\ln^\gamma\ln\calL^{-1}(t)\leq t \quad \mbox{and}\quad 
 t\leq 2^{1+\gamma}\ln\calL^{-1}(t)\ln^\gamma\ln\calL^{-1}(t)
\]
for sufficiently large $t$ and therefore 
\begin{equation}\label{eq:war mal ein lemmalein}
 %e^{f(t)}(f(t))^\gamma\leq 5t \quad \mbox{and}\quad
 e^{f(t)}(f(t))^\gamma \geq 2^{-1-\gamma} t \quad \mbox{for all } t>t_0, 
\end{equation}
with some $t_0>0$ and if we abbreviate $f(t):=\ln\ln\calL^{-1}(t)$. 
%\qed

We recall that the Lambert W function $W$ (see \cite{lambertfct}) is defined to be the inverse of $0\leq x\mapsto xe^x$. An inverse of $0\leq x\mapsto x^\gamma e^x$ is given by $0\leq x\mapsto \gamma W(\frac1\gamma x^{\frac1\gamma})$ and we see from \eqref{eq:war mal ein lemmalein} that, given any $c>0$, there are $t_0>0$ and $c_1>0$ such that 
\begin{equation}\label{eq:fgeqW}
 f(ct)\geq \gamma W\left(\frac1\gamma(c_1 t)^{\frac1\gamma}\right) \qquad \mbox{for all } t>t_0.
\end{equation}

By taking the logarithm on both sides of the defining relation for $W$, namely in $W(z)e^{W(z)}=z$ for $z>0$, we obtain 
\begin{equation}\label{eq:logLambert}
 W(z)+\ln W(z)=\ln z \qquad \text{for } z>0
\end{equation}
and since $W(z)>1$ for $z>e$, this entails $W(z)<\ln z$ for $z>e$. We can use this estimate in \eqref{eq:logLambert} to infer that 
\[
 W(z)=\ln z -\ln W(z)  > \ln z -\ln\ln z \qquad \mbox{for all } z>e.
\]
Employing this together with \eqref{eq:fgeqW}, we can find positive constants $c_2$, $c_3$, $c_4$, and $t_1>0$, such that 
\begin{align*}
 f(ct)\geq& \gamma W\left(\left(c_2t\right)^{\frac1\gamma}\right)\geq \gamma \ln\left(\left(c_2 t\right)^{\frac1\gamma}\right)-\gamma \ln\left(\ln\left(\left(c_2t\right)^{\frac1\gamma}\right)\right)\\
 =&\ln\left(c_2 t \ln^{-\gamma}\left(\left(c_2t\right)^{\frac1\gamma}\right)\right)=\ln\left(c_3t\ln^{-\gamma}\left(c_2 t\right)\right)\\
 =&\ln\left(c_3t\left(\ln c_2+\ln t\right)^{-\gamma}\right)\geq \ln \left(c_4t\ln^{-\gamma}t\right) \qquad \mbox{ for all } t>t_1.
\end{align*}
Accordingly, 
\begin{equation}\label{eq:Linversgeq}
 \calL^{-1}(ct)=e^{e^{f(ct)}}\geq e^{e^{\ln\left(c_4t\ln^{-\gamma} t\right)}}=e^{c_4t\ln^{-\gamma}t} \qquad \mbox{ for } t>t_1.
\end{equation}
Furthermore, there is $\xi_0>0$ such that 
\begin{equation}\label{eq:Lstrichestimate}
 \calL'(\xi)=\frac1\xi\ln^{\kappa\frac{n+2}n}\ln M\xi\leq \frac1{\sqrt \xi} \qquad \mbox{for all } \xi>\xi_0.
\end{equation}
If we combine \eqref{eq:Linversgeq} and \eqref{eq:Lstrichestimate} with the monotonicity of $\frac1{\calL'}$ guaranteed by \eqref{ell:monLstrich}, we can find $t_2>0$ such that 
\[
 (\calL^{-1})'(ct)=\frac{1}{\calL'(\calL^{-1}(t))}\geq \frac1{\calL'(e^{c_4 t\ln^{-\gamma}t})}\geq e^{\frac{c_4}2 t\ln^{-\gamma} t} \qquad \mbox{for all }t>t_2.
\]
In conclusion, together with Theorem \ref{rapid} \ref{r_Egeq} this proves Corollary \ref{cor:secondexample}.
% , if we let $M>C_0ee^{-\alpha}$ and choose $\xi _2\in[1,\frac Me)$ such that 
% \[
%  \sup_{x\in\R^n} u_0(x) \leq C_0e^{-\alpha e^{|x|^\gamma}}\leq C_0e^{-\alpha}<\xi _2<\xi _2^{\frac{2}{1+q_0}}<\xi _2^2 
% \]
% for some $q_0\in(0,1)$. Our definition of $\kappa$ and the condition on $\gamma$ ensure that $\kappa\gamma=\gamma^2\frac{n}{n+2}>1$ and thus $\irn \frac1{1+(|x|^\gamma)^\kappa}$ is finite and hence so is $\irn \ell\circ u_0$.
\end{proof}

% given any $\eta>0$ there is $t_1>0$ such that 
% \[
%  (1-\eta)\beta \ln(\frac1\beta(c_1t)^{\frac1\beta})\leq \beta W(\frac1\beta(c_1 t)^{\frac1\beta})\leq f(t)\leq\beta W(\frac1\beta(c_2 t)^{\frac1\beta})\leq(1+\eta)\beta\ln(\frac1\beta(c_2t)^{\frac1\beta})
% \]
% whenever $t>t_1$. Accordingly, for any $\eta>0$ there are $c_3>0$ and $c_4>0$ such that 
% \[
%  \calL^{-1}(t)=e^{e^{f(t)}}\geq e^{c_3 t^{1-\eta}} \quad \mbox{and}\quad \calL^{-1}(t)\leq e^{c_4 t^{1+\eta}}
% \]
% for sufficiently large $t$.
% 
% The right hand sides of the estimates for $\E$ in Lemma \ref{lem:rapid_Egeq} and Lemma \ref{lem:rapid_Eleq} therefore turn into 
% \[
%  \ln\frac1{\calL'(\calL^{-1}(ct)} \geq\ln\frac1{\calL'(e^{c_5 t^{1-\eta}})} = \ln \left(e^{c_5 t^{1-\eta}} \ln^{-\kappa\frac{n+2}n} \ln Me^{c_5 t^{1-\eta}}\right) = c_5 t^{1-\eta} -\frac{\kappa(n+2)}n \ln\ln\ln Me^{c_5t^{1-\eta}}
% \]
% or 
% \[
%  \ln\frac1{\calL'(\calL^{-1}(ct)}\leq c_6t^{1+\eta}-\frac{\kappa(n+2)}n \ln\ln\ln Me^{c_6t^{1+\eta}},
% \]
% respectively, for sufficiently large values of $t$ and with some $c_5, c_6>0$ (depending on $\eta$).

%
%
%
%
%
%
%
%
%
%
%
\end{document}